\documentclass[11pt]{amsart}

\usepackage{amsfonts, enumerate} 
\usepackage{amssymb}
\usepackage[final]{graphicx, graphics,stmaryrd} 

\usepackage[margin=3cm]{geometry} 
\usepackage[english]{babel}
\usepackage[utf8]{inputenc}
\usepackage{amsthm}
\usepackage{graphics}
\usepackage{amsmath}
\numberwithin{equation}{section} 
\usepackage{amstext}

\usepackage[safe,extra]{tipa}
\usepackage{multirow}

\usepackage[colorlinks=true,urlcolor=blue,
citecolor=red,linkcolor=blue,linktocpage,pdfpagelabels,
bookmarksnumbered,bookmarksopen]{hyperref}
\usepackage[hyperpageref]{backref}

\usepackage{bm} 
\usepackage{stackrel} 
\usepackage{tikz-cd} 
\usepackage{mathtools} 

\newtheorem{teo}{Theorem}[section]

\newtheorem{lemma}[teo]{Lemma}

\theoremstyle{definition}
\newtheorem{defin}[teo]{Definition}

\def\Oo{\mathcal O}
\def\R{\mathbb{R}}

\def\N{\mathbb N}

\def\G{\mathbb{G}}

\newcommand{\m}{\mbox}

\newcommand{\cor}{\textit}

\newcommand{\fine}{\qed\newline}

\DeclareMathOperator{\supp}{supp}

\newcommand{\de}{\partial}

\makeatletter
\@namedef{subjclassname@2020}{\textup{2020} Mathematics Subject Classification}
\makeatother

\begin{document}
 	\title[Crtitical Concave-convex problems in Carnot groups]{Critical concave-convex problems in Carnot groups}

 	\author[M.\,Galeotti]{Mattia Galeotti}
 	\author[E.\,Vecchi]{Eugenio Vecchi}

 	\address[M.\,Galeotti]{Dipartimento di Matematica
 		\newline\indent Università  di Bologna \newline\indent
 		Piazza di Porta San Donato 5, 40126 Bologna, Italy}
 	\email{mattia.galeotti4@unibo.it}
 	
 	\address[E.\,Vecchi]{Dipartimento di Matematica
 		\newline\indent Università  di Bologna \newline\indent
 		Piazza di Porta San Donato 5, 40126 Bologna, Italy}
 	\email{eugenio.vecchi2@unibo.it}
 	
 	
 	
 	\date{\today}
 	
 	\thanks{E.V. and M.G.are
 		members of GNAMPA-Indam. E.V. is
 		partially 
 		supported by the PRIN 2022 project 2022R537CS \emph{$NO^3$ - Nodal Optimization, NOnlinear elliptic equations, NOnlocal geometric problems, with a focus on regularity}, founded by the European Union - Next Generation EU and by the Indam-GNAMPA project CUP E5324001950001 - {\em Problemi singolari e degeneri: esistenza, unicità e analisi delle proprietà qualitative delle soluzioni}.
 		M.G. is supported by the National Recovery and Resilience Plan (NRRP), research funded by the European Union – NextGenerationEU.
 		CUP D93C22000930002, {\em A multiscale integrated approach to the study of the nervous system in health and disease (MNESYS)}}

 \begin{abstract}
 We consider a model Dirichlet problem with concave-convex and critical nonlinearity settled in Carnot groups. Our aim is to prove the existence of two positve solutions in the spirit of a famous result by Ambrosetti, Brezis and Cerami. To this aim we use a variational Perron method combined with proper estimates of a family of functions which are minimizers of the relevant Sobolev inequality. Due to the lack of boundary regularity, we also have to be careful while proving that the first solution found is a local minimizer in the proper topology.
 \end{abstract}

\maketitle

\section{Introduction}
Let $\mathbb{G}$ be a Carnot group and let $\Omega \subset \mathbb{G}$ be an open, bounded and connected set with Lipschitz boundary $\partial \Omega$. Let $q \in (0,1)$, let $2^{\star}_{Q}:=\tfrac{2Q}{Q-2}$ be the critical Sobolev exponent related to the Sobolev inequality in $\mathbb{G}$, and let $\lambda >0$. We consider the following Dirichlet boundary value problem
\begin{equation}\tag{{$\mathrm{P}_{\lambda}$}}\label{EqProblem}
	\left\{\begin{array}{rll}
		-\Delta_{\mathbb{G}}u& = \lambda\, u^{q} + u^{2^{\star}_{Q}-1} & \textrm{ in } \Omega,\\
		u&>0 & \textrm{ in } \Omega,\\
		u&=0 & \textrm{ on } \partial \Omega.
	\end{array}\right.
\end{equation}

We stress that $-\Delta_{\mathbb{G}}$ denotes here the sub-laplacian on $\mathbb{G}$ which is a second-order differential operator with non-negative characteristic form that can be explicitly expressed as a sum of squares of vector fields satisfying the H\"{o}rmander condition, see e.g. \cite{Hormander}. We refer to Section~\ref{sec:Prel} for more details, including the Folland-Stein Sobolev spaces we will work with.\\
%
We immediately state the main result of this paper. In what follows,
we refer to Definition \ref{def:weak_sub_super_sol} for the precise definition of \emph{weak solution}
of \eqref{EqProblem}.
\begin{teo} \label{thm:main}
	Let $\Omega\subset\mathbb{G}$ be an open, bounded and connected
	set
	with Lipschitz boundary $\partial \Omega$, and let $q\in (0,1)$. Then, there exists $\Lambda > 0$
	such that
	\begin{itemize}
		\item[A)]problem \eqref{EqProblem} does not admit weak solutions
		for every $\lambda>\Lambda$;
		\item[B)] problem \eqref{EqProblem} admits at least one weak solution for every $\lambda \in (0,\Lambda]$;
		\item[C)] problem \eqref{EqProblem} admits at least two weak solutions for every $0<\lambda<\Lambda$.
	\end{itemize}
\end{teo}

The above theorem is the natural generalization to Carnot groups of classical results of \cite{ABC}. We refer e.g. to \cite{BESS,CCP,CoPe,BV} for further generalizations.

The interest in studying existence of positive solutions to critical problems in the Carnot group setting,
 is in the geometric significance of the purely critical problem in the model case of the Heisenberg group.
 Indeed, when $\lambda =0$ and $\Omega = \mathbb{H}^n$, the problem \eqref{EqProblem} becomes the famous CR-Yamabe problem studied by Jerison and Lee \cite{JerisonLee,JerisonLee2,JerisonLee3}. The problem  we are interested in is settled on bounded domains, where tipycally one can prove non-existence of positive solutions, at least in star-shaped domains, by appealing suitable versions of the Pohozaev identity. Because of this, the seminal paper by Brezis and Nirenberg \cite{BN} showed that adding a perturbative term, linear in \cite{BN}, but subsequently extended to much more general perturbations, may allow to prove the existence of one or more positive solutions. A crucial tool in the argument performed in \cite{BN} is provided by the use of the Aubin-Talenti functions, whose analogue in $\mathbb{H}^n$ made its appearance in \cite{JerisonLee2}. This was a key ingredient which gave rise to a prolific study of critical problems in $\mathbb{H}^n$, see e.g. \cite{Citti, GaLa, Ugu1, CitUg, MaUg, FelliUgu, MaMaPi, PaPiTe}.

As long as one needs explicit knowledge of proper replacements of the Aubin-Talenti functions, the only other sub-Riemannian structure where they are known is that of groups of Iwasawa type, see \cite{GaroVa2, GaroVa}. As far as we know, there are no other structures, nor Sobolev inequalities with $p\neq 2$, for which the minimizers are explicitly known. 
On the other hand, since the best constant in the Sobolev inequality is achieved in all Carnot groups (see \cite{GaroVa2}), it has been proved to be enough to know  the asymptotic behaviour at infinity of the minimizers. This is now known for $p\neq 2$ as well, see \cite{Loiudice3}, and it paved the way for a series of 
existence, multiplicity or non-existence of positive solutions for critical problems  à la Br\'{e}zis-Nirenberg in $\mathbb{G}$: we refer e.g. to \cite{BoUg, Loiudice1, Loiudice2, Loiudice4, BiGaVe}.

\medskip

Let us now briefly describe the proof of Theorem \ref{thm:main}: 
\begin{itemize}
	\item in Theorem \ref{thm:Existence_First} we prove the existence of a first solution by means of a variational Perron method which transfers the approach of Struwe \cite{Struwe} to the Carnot group setting. In particular, setting
	\begin{equation*}
	\Lambda := \sup \{ \lambda >0: \eqref{EqProblem} \textrm{ admits a weak solution}\},
\end{equation*}
	\noindent we show first that $0<\Lambda<+\infty$, and this immediately provides a threshold for the non-existence of weak solutions. Once this is done, we use the unique solution of the purely sublinear problem \eqref{eq:Sublinear_Problem} as a weak subsolution and we construct a weak supersolution for fixed $\lambda$ by using the weak solution for a bigger $\lambda'$; 
	\item we show that for $\lambda \in (0,\Lambda)$ the first solution obtained as described before is a local minimizer in the natural topology associated with problem \eqref{EqProblem}, see Lemma \ref{lem:locmin}. We stress here that in \cite{ABC} the authors made use of a famous result by Brezis and Nirenberg \cite{BNH1C1} which does not have an analog in the Carnot group setting. This is due to the fact that $C^{1,\alpha}$ regularity up to the boundary is still a delicate issue at the so called characteristic points: the first obstructions have been observed by Jerison \cite{Jerison,Jerison2}, but this is still an active field of research, see e.g. \cite{BaCiCu,BaGaMu,AbTr}. For this reason we follow here a more variational approach based on a paper by Alama \cite{Alama}, already used in a different setting in \cite{AbDiVa};
	\item we prove the existence of a second solution following an argument originally due to Tarantello \cite{Tarantello}: this combines the Ekeland variational principle \cite{Ekeland} with the fine asymptotic expansions proved in \cite{Loiudice1}.
\end{itemize} 

\medskip

We stress that the multiplicity result obtained in Theorem \ref{thm:main} can be easily extended to cover the convex-case of a Sobolev sub-critical nonlinearity.

\medskip

The paper is organized as follows: in Section \ref{sec:Prel} we recall the basic facts on Carnot groups and we set the variational functional setting necessary for the study of \eqref{EqProblem}. We also recall the basic result regarding the purely sublinear problems, like existence and uniqueness of a positive solution and a comparison principle resembling the classical one. In Section \ref{sec:First_Solution} we prove the existence of a first solution as described before, while the existence of a second solution (for $\lambda \in (0,\Lambda)$) is postponed to the final Section \ref{sec:Second_Solution}.

\section{Preliminaries}\label{sec:Prel}
In this section we collect all the relevant notations, definitions and preliminaries needed in the rest of the paper.
\subsection{Carnot groups}
A Carnot group $\mathbb{G}=(\mathbb{R}^{N},\diamond)$ of step $k$ is a connected, simply connected Lie group whose finite dimensional Lie algebra $\mathfrak{g}$ of left-invariant (w.r.t. $\diamond$) vector fields admits a stratification of step $k$, namely there exist $k$ linear subspaces $\mathfrak{g}_1, \ldots, \mathfrak{g}_k$ such that
\begin{equation*}
	\mathfrak{g} = \mathfrak{g}_{1}\oplus \ldots \oplus \mathfrak{g}_{k}, \qquad [\mathfrak{g}_1,\mathfrak{g}_i]=\mathfrak{g}_{i+1}, \qquad \mathfrak{g}_{k}\neq \{0\}, \qquad \mathfrak{g}_i = \{0\} \textrm{ for all } i>k.
\end{equation*}
In particular, this implies that Carnot groups are a special instance of graded groups.\\
We call $\mathfrak{g}_1$ the horizontal layer. We denote by $X_1, \ldots, X_{N}$ a basis of left-invariant vector fields of $\mathfrak{g}$ such that the following holds:
\begin{itemize}
	\item $X_1, \ldots, X_{m_1}$ is an orthonormal basis of $\mathfrak{g}_{1}$ w.r.t. the scalar product $\langle \cdot, \cdot \rangle_{\mathfrak{g}_{1}}$;
	\item for every $1<i\leq k$, $X_{m_{i-1}+1},\ldots, X_{m_{i}}$ is a basis of of $\mathfrak{g}_{i}$;
	\item $m_{0}=0$ and $n_i := m_{i}-m_{i-1} = \dim \mathfrak{g}_i$ for every $1\leq i\leq k$;
	\item $m_1 + \ldots + m_k = N$.
\end{itemize}
We define the homogeneous dimension of $\mathbb{G}$ as
\begin{equation*}
	Q:= \sum_{i=1}^{k}i \cdot n_{i}.
\end{equation*}
The left translations $\tau : \mathbb{G}\to \mathbb{G}$ provides a family of automorphisms of $\mathbb{G}$, and are defined as follows
\begin{equation}\label{eq:left-translation}
	\tau_{h}(g):= h \diamond g, \quad \textrm{ for a given $h\in \mathbb{G}$,} 
\end{equation}
The anisotropic dilations $\delta_{\lambda}: \mathbb{G}\to \mathbb{G}$ of $\mathbb{G}$ are instead defined as 
\begin{equation}\label{eq:dilation}
	\delta_{\lambda}(g) = \left(\lambda^{\alpha_1}g_1, \ldots, \lambda^{\alpha_{N}}g_N \right), \quad \textrm{ for every } \lambda >0,
\end{equation}
\noindent where $\alpha_{j} = i$ if $m_{i-1}<j\leq m_{i}$.
We notice that $Q = \alpha_1 + \ldots + \alpha_{N}$.\\

\medskip


Given a smooth horizontal vector field $V=v_1 X_1 + \ldots + v_{m_1}X_{m_1}$, we define its horizontal divergence as
\begin{equation}\label{eq:horizontal_divergence}
	\mathrm{div}_{\mathbb{G}}V := X_{1}v_1 + \ldots + X_{m_1}v_{m_1}.
\end{equation}
Moreover, given a smooth enough scalar-valued function $u:\mathbb{G}\to \mathbb{R}$, we can define
the horizontal gradient of $u$ as
\begin{equation}\label{eq:Horizontal_grad}
	\nabla_{\mathbb{G}}u := (X_{1}u, \ldots, X_{m_1}u),
\end{equation}
\noindent and the sub-Laplacian of $u$ as
\begin{equation}\label{eq:horizontal_lap}
	\Delta_{\mathbb{G}}u := \mathrm{div}_{\mathbb{G}}(\nabla_{\mathbb{G}}u) = X_1^2 u + \ldots + X_{m_1}^{2} u,
\end{equation}

The Lebesgue measure $\mathcal{L}^{N}$ coincides with the Haar measure of $\mathbb{G}$ and hence is left-invariant and satisfies the following scaling property:
\begin{equation}
	\mathcal{L}^{N}(\delta_{\lambda}(E)) = \lambda^{Q} \mathcal{L}^{N}(E) \quad \textrm{ for every measurable set } E \subset \mathbb{G}.	
\end{equation}
Every integral in this manuscript has to be understood
with respect to the Haar measure, unless otherwise stated.

\medskip

Every Carnot group can be endowed with several homogeneous norms. A homogeneous (quasi)norm $\rho : \mathbb{G} \to \mathbb{R}$ is a non-negative function further satisfying the following properties:
\begin{itemize}
	\item $\rho(g)=0$ if and only if $g=0$;
	\item $\rho(\delta_{\lambda}(g)) = \lambda \, \rho(g)$ for every $g \in \mathbb{G}$ and for every $\lambda >0$;
	\item $\rho(h\diamond g) \leq C \left(\rho(h) + \rho(g)\right)$ for every $g,h \in \mathbb{G}$ and for some constant $C \geq 1$.
\end{itemize}
By a famous result of Folland \cite{Folland}, there exists a homogeneous norm $|\cdot|_{\mathbb{G}}$ on $\mathbb{G}$ and a positive constant $C_Q>0$, depending only on $Q$, such that the function
\begin{equation}\label{eq:def_Gamma}
	\Gamma_{h}(g):= \dfrac{C_Q}{|h^{-1}\diamond g|^{Q-2}_{\mathbb{G}}}, \quad \textrm{ with } Q \geq 3,
\end{equation}
\noindent is a fundamental solution of $-\Delta_{\mathbb{G}}$ with pole at $h \in \mathbb{G}$.
Moreover, homogeneous norms can be used to define distances as follows: 
\begin{equation*}
	d_{\rho}(g,h):= \rho(h^{-1}\diamond g).
\end{equation*}
We stress that there are other possible choices of distances (i.e. the so called $CC$-distance and many others) which are all equivalent.
%
%
Finally, we will denote by 
\begin{equation*}
	B_{r}(g_0) := \{ g \in \mathbb{G}: d_{\mathbb{G}}(g,g_0) = |g_{0}^{-1}\diamond g|_{\mathbb{G}} <r\},
\end{equation*} 
\noindent the open ball of radius $r>0$ and center $g_{0}\in \mathbb{G}$.\\
We refer e.g. to \cite{BLU} for a comprehensive introduction to the subject.

\medskip

\subsection{The functional setting}\label{sec:functional_setting}
Let $\mathcal{O} \subseteq \mathbb{G}$ be an open set.
For every $f \in C^{\infty}_{0}(\mathcal{O})$ there exists a positive constant $C_Q>0$ depending only on the homogeneous dimension $Q$ such that the following Sobolev inequality holds true
\begin{equation}\label{eq:Sobolev_Ineq}
	\|f\|^{2}_{L^{2_{Q}^{\star}}(\mathcal{O})} \leq C_{Q} \, \| |\nabla_{\mathbb{G}}f| \|^{2}_{L^{2}(\mathcal{O})},
\end{equation}
\noindent where 
\begin{equation}
	2_{Q}^{\star} := \dfrac{2Q}{Q-2},
\end{equation}
\noindent denotes the (sub-elliptic) critical Sobolev exponent.
Thanks to \eqref{eq:Sobolev_Ineq}, 
$\| |\nabla_{\mathbb{G}}f| \|_{L^{2}(\Omega)}$ provides a norm on the space $C^{\infty}_{0}(\Omega)$.
We define the Folland-Stein space $S^{1}_{0}(\mathcal{O})$ as the completion of $C^{\infty}_{0}(\mathcal{O})$ w.r.t.\,the above norm, and we set
$$\|u\|_{S^{1}_{0}(\mathcal{O})} = \||\nabla_\mathbb{G}u|\|_{L^2(\mathcal{O})}
\quad\text{for every $u\in S^{1}_{0}(\mathcal{O})$}.$$
We explicitly observe that, owing to \eqref{eq:Sobolev_Ineq}, we have
\begin{equation} \label{eq:explicitS01def}
 S_0^1(\mathcal{O}) = \big\{u\in L^{2_{Q}^{\star}}(\mathcal{O}):\,\text{$X_iu\in L^2(\mathcal{O})$
 for all $1\leq i\leq m_1$}\big\},
\end{equation}
where $X_1u,\ldots,X_{m_1}u$ are meant in the sense of distributions.\\
We now remind a couple of basic properties of $S_0^1(\mathcal{O})$ when $\mathcal{O}$ is an open and bounded set, see e.g. to \cite{FollandStein}.
\begin{itemize}
 \item $S_0^1(\mathcal{O})$ is endowed with a structure of real Hilbert space by the inner product
 $$\langle u,v\rangle_{S_0^{1}(\mathcal{O})} = \int_{\mathcal{O}}\langle \nabla_\mathbb{G}u,
 \nabla_\mathbb{G}v\rangle_{\mathfrak{g}_{1}}\qquad (u,v\in S_0^1(\mathcal{O})),$$
whose associated norm is precisely $\|\cdot\|_{S_0^1(\mathcal{O})}$.

 \item $S_0^1(\mathcal{O})$ is continuously embedded into $L^p(\mathcal{O})$ for every $1\leq p\leq 2_Q^\star$.
 Furthermore, this embedding turns out to be \emph{compact} when $1\leq p < 2_Q^\star$. 
 \end{itemize}
We refer e.g. to \cite{FollandStein} for more details.

\medskip

We are now ready to properly set the definition of weak sub/supersolution of \eqref{EqProblem}.
\begin{defin}\label{def:weak_sub_super_sol}
	Let $\Omega\subseteq\mathbb{G}$ be an open, bounded and connected set.
We say that a function $u \in S^{1}_{0}(\Omega)$ is a weak subsolution (resp. supersolution) of \eqref{EqProblem} if it satisfies the following properties:
	\begin{itemize}
		\item[(i)] $u>0$ in $\Omega$.
		\item[(ii)] For every $0\leq \varphi \in C^{\infty}_{0}(\Omega)$, it holds that
		\begin{equation}\label{eq:weak_sub_super_sol}
			\int_{\Omega}\langle \nabla_{\mathbb{G}}u, \nabla_{\mathbb{G}}\varphi \rangle_{\mathfrak{g}_{1}} \leq (\textrm{resp. } \geq)  \int_{\Omega}\left(\lambda u^{q}  + \ u^{2^{\star}_{Q}}\right)\varphi.
		\end{equation}
	\end{itemize}
	Finally, we say that $u \in S^{1}_{0}(\Omega)$ is a weak solution of \eqref{EqProblem} if it is both a weak subsolution and a weak supersolution of \eqref{EqProblem} without the non-negativity condition on $\varphi$.
\end{defin}

\medskip

Let us close this section recalling a few results on the Sobolev inequality \eqref{eq:Sobolev_Ineq}.
\begin{lemma}
	The best Sobolev constant in \eqref{eq:Sobolev_Ineq} (with $\mathcal{O} = \mathbb{G}$) is achieved by a positive function $T \in S^{1}_{0}(\mathbb{G})$, and it is characterized as follows
	\begin{equation}\label{eq:def_best_constant}
		S_{\mathbb{G}} := \inf_{f\in S^{1}_{0}(\mathbb{G})} \dfrac{\| |\nabla_{\mathbb{G}}f| \|^{2}_{L^{2}(\mathbb{G})}}{	\|f\|^{2}_{L^{2_{Q}^{\star}}(\mathbb{G})}}.
	\end{equation}
	Up to a constant, $T \in S^{1}_{0}(\mathbb{G})$ is also a weak solution to 
	\begin{equation}\label{eq:Critical_PDE_in_G}
		-\Delta_{\mathbb{G}}u = u^{2^{\star}_{Q}-1} \quad \textrm{in } \mathbb{G}.
	\end{equation}
	Moreover, the following holds:
	\begin{itemize}
		\item there exists a positive constant $M_1>0$ such that 
		\begin{equation}\label{eq:Stima_Bonf_Ugu}
			T(g) \leq M_1 \, \min \{1, |g|_{\mathbb{G}}^{2-Q}\}, \quad \textrm{ for every } g \in \mathbb{G},
		\end{equation} 
		\item there exists a positive constant $M_2 >0$ such that
		\begin{equation}\label{eq:Stima_Loiudice}
			T(g) \geq M_2 \, \dfrac{|B_{1}(0)|}{(1+|g|_{\mathbb{G}})^{Q-2}}, \quad \textrm{ for every } g \in \mathbb{G}.
		\end{equation}
	\end{itemize}
	\begin{proof}
		The fact that the best Sobolev constant in \eqref{eq:Sobolev_Ineq} (with $\mathcal{O} = \mathbb{G}$) is achieved has been proved in \cite{GaroVa}. We refer to \cite[Theorem 3.4]{BoUg} for a proof of \eqref{eq:Stima_Bonf_Ugu} and to  \cite[Lemma 3.2]{Loiudice1} for a proof of \eqref{eq:Stima_Loiudice}.
	\end{proof}
\end{lemma}

As it is well known from the seminal paper \cite{BN}, a major role in finding solutions to critical problems is played by a suitable localized version of the family of minimizers of \eqref{eq:Sobolev_Ineq}. To be more precise, let $T \in S^{1}_{0}(\mathbb{G})$ be a minimizer of \eqref{eq:def_best_constant}. For every $\varepsilon >0$ define the rescaled function
\begin{equation}\label{eq:Minimi_riscalati}
	T_{\varepsilon}(g) := \varepsilon^{(2-Q)/2} T \left(\delta_{1/\varepsilon}(g)\right).
\end{equation}
Let further be $R>0$ such that $B_{R}(0)\subset \Omega$ and let $\varphi \in C^{\infty}_{0}(B_{R}(0))$ be a cut-off function such that $0\leq \varphi \leq 1$ and $\varphi \equiv 1$ in $B_{R/2}(0)$. Finally, define the family of functions 
\begin{equation}\label{eq:def_u_varepsilon}
	U_{\varepsilon}(g):= \varphi(g) T_{\varepsilon}(g), \quad g \in \mathbb{G}.
\end{equation}
Now, we have the following

\begin{lemma}\label{lem:Stime_Talentiane_modificate}
	Let $T_{\varepsilon}$ and $U_{\varepsilon}$ be as above.
	The following holds:
	\begin{itemize}
			\item[i)] Due to scaling invariance, 
		\begin{equation}
			\| |\nabla_{\mathbb{G}}T_{\varepsilon}|\|^{2}_{L^{2}(\mathbb{G})} = \| T_{\varepsilon}\|^{2^{\star}_{Q}}_{L^{2_{Q}^{\star}}(\mathbb{G})} = S_{\mathbb{G}}^{Q/2}.
		\end{equation}
		\item[ii)] The function $U_{\varepsilon}$ satisfies the following estimates as $\varepsilon \to 0^{+}$
		\begin{align}
			\| |\nabla_{\mathbb{G}}U_{\varepsilon}|\|^{2}_{L^{2}(\mathbb{G})} & =  S_{\mathbb{G}}^{Q/2} + O(\varepsilon^{Q-2}) \label{eq:Stima_grad_u_varepsilon}\\
			\|U_{\varepsilon}\|^{2^{\star}_{Q}}_{L^{2_{Q}^{\star}}(\mathbb{G})}& = S_{\mathbb{G}}^{Q/2} + O(\varepsilon^{Q}) \label{eq:Stima_u_varepsilon}. 
		\end{align}
	\end{itemize}	
	\begin{proof}
		We refer to \cite[Lemma 3.3]{Loiudice1} for a proof of both \eqref{eq:Stima_grad_u_varepsilon} and \eqref{eq:Stima_u_varepsilon}.
	\end{proof}
\end{lemma}

\medskip

\subsection{Auxiliary sublinear problem}\label{subsec:Auxiliary}
A major role in what follows will be played by the weak solution of the following problem
%
\begin{equation}\label{eq:Sublinear_Problem}
	\left\{\begin{array}{rll}
				-\Delta_{\mathbb{G}}u &= \lambda\, u^{q} & \textrm{ in } \Omega,\\
				u&>0 & \textrm{ in } \Omega,\\
				u&=0 & \textrm{ on } \partial \Omega,
	\end{array}\right.
\end{equation}
\noindent where $q \in (0,1)$ and $\lambda >0$. We say that $u \in S^{1}_{0}(\Omega)$ is a weak solution of \eqref{eq:Sublinear_Problem} analogously to Definition \ref{def:weak_sub_super_sol}.
In this context, we have the following

\begin{teo}\label{thm:Sublinear_Problem}
	Let $\Omega\subseteq\mathbb{G}$
	be a bounded open set. Moreover, let $q \in (0,1)$ and $\lambda >0$. 
	Then, there exists a unique weak solution $\underline{u}_{\lambda} \in S^{1}_{0}(\Omega) \cap L^{\infty}(\Omega)$ to \eqref{eq:Sublinear_Problem}. Moreover, $\underline{u}_{\lambda}$ is a global minimizer in the $S^{1}_{0}(\Omega)$-topology
	of the functional 
		\begin{equation}\label{eq:def_Functional_Singular}
				J_{\lambda}(u):= \dfrac{1}{2}\int_{\Omega}|\nabla_{\mathbb{G}}u|^{2} 
				-\dfrac{\lambda}{q+1} \int_{\Omega}|u|^{q+1}.
			\end{equation}
		Finally, we also have that $J_{\lambda}(\underline u_\lambda)<0$.
   \end{teo}
The above result is actually a particular case of \cite[Theorem 1.1, Proposition 5.1]{BPV}, where Brezis-Oswald-type results have been proved for more general H\"{o}rmander operators.

\medskip

We now state an {\em ad hoc} comparison principle for weak super and subsolutions of the model purely sublinear problem \eqref{eq:Sublinear_Problem}. This can be seen as a particulr instance, in the Carnot group setting, of \cite[Lemma 3.3]{ABC}, which in turn was inspired by \cite{BrezisKamin}. We refer to \cite{Ruzhansky_Suragan} for more general result whose proofs rely on the validity of suitable Picone-type inequalities.

	\begin{lemma}\label{lem:Weak_Comparison_Model}
		Let $\lambda >0$, $q\in (0,1)$ and $v,w\in S^1_0(\Omega)$ weakly satisfy
		\begin{equation}\label{eq:per_v}
		   \left\{\begin{array}{rll}
		   	-\Delta_\G v &\leq \lambda  v^q & \textrm{ in }\Omega,\\
		    v& >0 & \textrm{ in } \Omega,\\
		    v&=0 & \textrm{ on } \partial \Omega,
		   \end{array}
		   \right.
		\end{equation}
		\noindent and 
			\begin{equation}\label{eq:per_w}
			\left\{\begin{array}{rll}
				-\Delta_\G w& \geq \lambda w^q & \textrm{in }\Omega,\\
				w &>0 & \textrm{in }\Omega,\\
				w&=0 & \textrm{on } \partial \Omega.
			\end{array}
			\right.
		\end{equation}
		 Then $w\geq v$ in $\Omega$.
		\begin{proof}
			We closely follow \cite[Proof of Lemma 3.3]{ABC}. 
			We choose first a smooth function $\theta\in C^\infty(\mathbb{R})$ satisfying the following properties
			\begin{itemize}
				\item  $\theta(t) = 0$ for $t \leq 0$ and $\theta(t) =1$ for $t \geq 1$;
				\item $\theta$ is non-decreasing on $\mathbb{R}$;
			\end{itemize}
			and we define (for every $\varepsilon > 0$)
			$$\theta_\varepsilon(t): = \theta\left(\frac{t}{\varepsilon}\right) \in S_0^1(\Omega).$$ 
			We further consider the variational formulation of both \eqref{eq:per_v} and \eqref{eq:per_w}, namely 
			\begin{equation}\label{eq:variational_for_w}
				\int_{\Omega}\langle \nabla_{\mathbb{G}}w, \nabla_{\mathbb{G}}\varphi \rangle_{\mathfrak{g}_1} \geq \lambda \int_{\Omega} w^q\varphi \quad 0\leq \varphi \in S^{1}_{0}(\Omega),
			\end{equation}
			\noindent and 
			\begin{equation}\label{eq:variational_for_v}
				\int_{\Omega}\langle \nabla_{\mathbb{G}}v, \nabla_{\mathbb{G}}\varphi \rangle_{\mathfrak{g}_1} \geq \lambda \int_{\Omega}v^q \varphi \quad 0\leq \varphi \in S^{1}_{0}(\Omega).
			\end{equation}
			We then test \eqref{eq:variational_for_w} with $\theta_{\varepsilon}(v-w) v$ and \eqref{eq:variational_for_v} with $\theta_{\varepsilon}(v-w) w$. Finally, we subtract the latter to the former, getting
			\begin{equation}\label{eq:Testata_con_theta}
			\begin{aligned}
			\int_{\Omega}&\left(w^{q-1}-v^{q-1} \right)\theta_{\varepsilon}(v-w)vw \\
			&= \int_{\Omega}\langle \nabla_{\mathbb{G}}w, \nabla_{\mathbb{G}}(v-w)\rangle_{\mathfrak{g}_1}v \, \theta'_{\varepsilon}(v-w) - \int_{\Omega}\langle \nabla_{\mathbb{G}}v, \nabla_{\mathbb{G}}(v-w)\rangle_{\mathfrak{g}_1}w \, \theta'_{\varepsilon}(v-w)\\
			&\leq \int_{\Omega}\langle \nabla_{\mathbb{G}}v,\nabla_{\mathbb{G}}(v-w)\rangle_{\mathfrak{g}_1} (v-w)\theta'_{\varepsilon}(v-w) \\
			&= \int_{\Omega}\langle \nabla_{\mathbb{G}}v, \nabla_{\mathbb{G}}(\gamma_{\varepsilon}(v-w))\rangle_{\mathfrak{g}_1} = \int_{\Omega}(-\Delta_{\mathbb{G}}v) \, \gamma_{\varepsilon}(v-w),
		\end{aligned}
		\end{equation}
			\noindent where 
			\begin{equation*}
				\gamma_{\varepsilon}(t):= \int_{0}^{t}s\theta'_{\varepsilon}(s)\, ds.
			\end{equation*}
			By construction of $\theta_{\varepsilon}$, it follows that $0\leq \gamma_{\varepsilon}(t)\leq \varepsilon$ for every $t\in \mathbb{R}$. Therefore, exploiting both \eqref{eq:Testata_con_theta}, \eqref{eq:per_v} and H\"{o}lder inequality, we find that
			\begin{equation}
					\int_{\Omega}\left(w^{q-1}-v^{q-1} \right)\theta_{\varepsilon}(v-w)vw \leq \lambda \, \varepsilon \int_{\Omega}v^q \leq C(\lambda, \|v\|_{L^{1}(\Omega)},|\Omega|,q)\,  \varepsilon.
			\end{equation}
			By letting $\varepsilon \to 0^+$, we find that
			\begin{equation*}
				\int_{\{v>w\}}\left(w^{q-1}-v^{q-1} \right)vw \leq 0,
			\end{equation*}
			\noindent from which, recalling that $q \in (0,1)$, we conclude that $|\{v>w\}|=0$. This closes the proof.
			\end{proof}
	\end{lemma}

%
%
%

We now state an extension to the Carnot group setting of \cite[Theorem 2.4]{Struwe}. We omit the proof recalling that it is a simplified version of \cite[Lemma 3.4]{BiGaVe} where the considered functional was not differentiable i n $S_{0}^{1}(\Omega)$.

\begin{lemma}\label{lem:Perron}
	Let $\underline{u}, \overline{u}\in S^{1}_{0}(\Omega)$ be a weak subsolution and a weak supersolution, respectively, of problem \eqref{EqProblem}. 
	We assume that 
	\begin{itemize}
		\item[a)] $\underline{u}(g) \leq \overline{u}(g)$ for a.e.\,$g\in \Omega$;
		\item[b)] for every open set $\Oo\Subset\Omega$ there exists
		$C = C(\Oo,\underline{u}) > 0$ such that
		$$\text{$\underline{u}\geq C$ a.e.\,in $\Oo$}.$$
	\end{itemize}
	Then, there exists a weak solution $u \in S^{1}_{0}(\Omega)$ of \eqref{EqProblem} such that 
	$$\text{$\underline{u}(g) \leq u(g) \leq \overline{u}(g)$ for a.e. $g \in \Omega$}.$$
\end{lemma}

\section{Proof of Theorem \ref{thm:main} - Part A) and part B)}\label{sec:First_Solution}
%
The goal of this Section is to prove the existence of a positive weak solution to \eqref{EqProblem}. 
To begin with, we define
\begin{equation}\label{eq:DefinitionLambda}
	\Lambda := \sup \{ \lambda >0: \eqref{EqProblem} \textrm{ admits a weak solution}\}.
\end{equation}
Our task now is pretty standard and it consists in the following steps: 
\begin{itemize}
	\item[I)] prove that $0<\Lambda<+\infty$;
	\item[II)] prove that problem \eqref{EqProblem} admits a weak solution for every $0<\lambda\leq \Lambda$.
\end{itemize}

\noindent We split the proof of I) in two lemmas. 
\begin{lemma}\label{lem:lambda0}
	Let $\Lambda$ as defined in \eqref{eq:DefinitionLambda}. Then $\Lambda >0$.
	\begin{proof}
		We will show that there exists a sufficiently small $\lambda >0$ such that 
		\eqref{EqProblem} has a solution. To this aim, we will use Lemma \ref{lem:Perron} exhibiting both a super and a subsolution.
		Looking for a supersolution, we consider the following auxiliary torsion problem
		\begin{equation}
			\left\{\begin{array}{rl}
				-\Delta_{\mathbb{G}} V = 1 & \textrm{in } \Omega,\\
				V = 0 & \textrm{on } \partial \Omega,
			\end{array}\right.
		\end{equation}
		\noindent whose unique solution is provided by Lax-Milgram Theorem. Moreover, by a classical Stampacchia iteration method, it holds that $V\in L^{\infty}(\Omega)$. Observe further that for any positive constant $C$,
		the function $C\cdot x^p-x$, with $p>1$, has negative value for $x>0$ sufficiently small.
		Therefore, for every $C,C'\in \R^+$
		there exists $\lambda^*>0$ such that 
		for every $\lambda<\lambda^*$, 
		\[
		\exists \, m_\lambda\in \R^+:\quad \lambda\cdot  C'\cdot m_\lambda^q+C\cdot m_\lambda^p-m_\lambda\leq0.
		\]
		We fix $\lambda<\lambda^*$ and set $C=\|V\|_{L^{\infty}(\Omega)}^p$, $C'=\|V\|_{L^{\infty}(\Omega)}^q$. 
		We define $\overline u_1:=m_\lambda V$, which weakly verifies
		\[
		\left\{\begin{array}{rll}
			-\Delta_{\G}\overline u_1&=m_\lambda \geq \lambda\overline u_1^q+\overline u_1^p &\ \m{in }\Omega\\
			\overline u_1 &>0 &\ \m{in }\Omega\\
			\overline u_1 &=0 &\ \m{on }\de\Omega,
			\end{array}\right.
		\]
				therefore it is a weak supersolution of \eqref{EqProblem}.\\
				Regarding the weak subsolution to \eqref{EqProblem}, we choose the unique solution $\underline u_\lambda$ to \eqref{eq:Sublinear_Problem}. 
				We can now conclude the proof by appealing Lemma \ref{lem:Perron}. Indeed, by Lemma \ref{lem:Weak_Comparison_Model} with $w=\overline{u}_1$ and $v=\underline{u}_{\lambda}$, we get that 
				$\overline{u}_1 \geq \underline{u}_{\lambda}$, which is condition a) of Lemma \ref{lem:Perron}. Regarding b) of Lemma \ref{lem:Perron} it is enough to recall \cite[Corollary 2.3]{BiGaVe}. This closes the proof.				
	\end{proof}
\end{lemma}

\begin{lemma}\label{lem:LambdaInf}
	Let $\Lambda$ as defined in \eqref{eq:DefinitionLambda}. Then, $\Lambda < +\infty$.
\end{lemma}
	\begin{proof}
	We consider the first eigenfunction $e_1$ of the operator $-\Delta_\G$
	with respect to the first Dirichlet eigenvalue $\mu_1$. In particular, the following characterization holds:
	\[
	\mu_1=\min\left\{\| |\nabla_\G u|\|^2_{L^2(\Omega)}:\ u\in S^1_0(\Omega)\m{ and }\|u\|^2_{L^2(\Omega)}=1\right\}.
	\]
	Knowing that $\|e_1\|_{L^2}=1$, $e_1>0$ a.e.~in $\Omega$ and $\| |\nabla_\G e_1|\|^2_{L^2}=\mu_1$,
	we have that any solution $u$ to \eqref{EqProblem} for some $\lambda$
	must verify
	\[
	\int_\Omega \langle \nabla_\G u,\nabla_\G e_1\rangle =\mu_1 \int_\Omega u e_1=\int_\Omega \lambda u^qe_1+u^pe_1.
	\]
	As $q\in (0,1)$ and $p>1$, for $\Lambda^*$ sufficiently big, we have
	\[
	\Lambda^*x^q+x^p>\mu_1x\quad\forall x\in \R^+.
	\]
	Therefore we must have $\lambda <\Lambda^*$ and this  proves $\Lambda\leq \Lambda^*<+\infty$.
\end{proof}

\medskip

Combining Lemma \ref{lem:lambda0} with Lemma \ref{lem:LambdaInf} we get I). Let us now turn our attention to the proof of II). Firstly, let us define the functional $I_{\lambda}$ naturally associated to \eqref{EqProblem}:
\begin{equation}\label{eq:Def_Ilambda}
	I_{\lambda}(u) := \dfrac{1}{2}\int_{\Omega}|\nabla_{\G}u|^2 - \dfrac{\lambda}{q+1}\int_{\Omega}|u|^{q+1} - \dfrac{1}{2^{\star}_{Q}}\int_{\Omega}|u|^{2^{\star_{Q}}}, \quad u\in S^{1}_{0}(\Omega).
\end{equation}

\begin{teo}\label{thm:Existence_First}
	Problem \eqref{EqProblem} admits at least
	one weak solution $u_\lambda\in S^{1}_{0}(\Omega)$ for every $\lambda \in (0,\Lambda]$.
\end{teo}
\begin{proof}
	As for the proof of Lemma \ref{lem:lambda0}, we find either a weak subsolution and a weak supersolution and then apply Lemma \ref{lem:Perron}. \\
	As long as the weak subsolution is concerned, we can take the unique solution $\underline{u}_{\lambda}$ of \eqref{eq:Sublinear_Problem}. 
	
	Let us now look for a weak supersolution. In doing this we profit of the very definition of $\Lambda$, which guarantess the existence of $\lambda' \in (\lambda, \Lambda)$ such that \eqref{EqProblem} (witht $\lambda =\lambda'$) admits a weak solution $u_{\lambda'}$. Clearly, this is a weak supersolution of \eqref{EqProblem}.\\
	By Lemma \ref{lem:Weak_Comparison_Model}, with $w = u_{\lambda'}$ and $v=w_{\lambda}$, it follows that
	\begin{equation}\label{eq:Claim}
		w_{\lambda}(g) \leq u_{\lambda'}(g), \quad \textrm{for a.e. } g\in \Omega.
	\end{equation}

We now set $\overline{u}= u_{\lambda'}$ and $\underline{u} = w_{\lambda}$,	
	and we apply Lemma \ref{lem:Perron}: this immediately yields that problem \eqref{EqProblem} admits a weak solution $u_{\lambda}$ for every $\lambda \in (0,\Lambda)$. Moreover, recalling the definition of $I_{\lambda}$ in \eqref{eq:Def_Ilambda}, such a solution satisfies that
	$$I_{\lambda}(u_\lambda) = \min\{u\in S_0^1(\Omega):\,w_{\lambda}\leq u\leq u_{\lambda'}\}
	\leq I_{\lambda}(w_{\lambda}).$$
	In particular,
	by Theorem \ref{thm:Sublinear_Problem} we have
	\begin{equation} \label{eq:Ilambdaulambdaneg}
		I_{\lambda}(u_{\lambda}) \leq I_{\lambda}(w_{\lambda}) \leq J_{\lambda}(w_{\lambda}) <0.
	\end{equation}
	It remains to consider the case $\lambda = \Lambda$. The proof
	is rather standard and pretty similar to that of \cite[Lemma 3.5]{BiGaVe}. We report it here for the sake of completeness.
	To begin with, we choose a monotone increasing sequence $\{\lambda_k\}_k\subseteq(0,\Lambda)$ such that 
     $\lambda_{k} \to \Lambda$ as $k\to+\infty$. Now, for each $k \in \mathbb{N}$, we set
	$$u_k := u_{\lambda_k}\in S^{1}_{0}(\Omega),$$
	\noindent where $u_{\lambda_k}$ is the weak solution of problem \eqref{EqProblem} (with $\lambda = \lambda_k$) constructed
	as above by means of Lemma \ref{lem:Perron}. Thanks to 
	\eqref{eq:Ilambdaulambdaneg}, for every $k\geq 1$ we have
	\begin{equation} \label{eq:Ilambdakneg}
		I_{\lambda_k}(u_{k}) 
		= \dfrac{1}{2} \int_{\Omega}|\nabla_{\mathbb{G}}u_k|^2 - \dfrac{\lambda_k}{q+1}\int_{\Omega}|u_k|^{q+1} - \dfrac{1}{2^{\star}_{Q}}\int_{\Omega}|u_k|^{2^{\star}_{Q}} <0.
	\end{equation}
	Moreover, by using $\varphi = u_k$ in \eqref{eq:weak_sub_super_sol},  and recalling that $u_k$
	solves \eqref{EqProblem} (with $\lambda = \lambda_k$), we get
	\begin{equation} \label{eq:testwithukzero}
		\int_{\Omega}|\nabla_{\mathbb{G}}u_k|^2 -\lambda_k\int_\Omega u_k^{q+1}
		-\int_\Omega u_k^{2^{\star}_{Q}} = 0.
	\end{equation}
	Combining \eqref{eq:Ilambdakneg} with \eqref{eq:testwithukzero},
	we notice that the sequence $\{u_k\}_k$ is  bounded in~$S^{1}_{0}(\Omega)$.
	Therefore, we can find a function
	$$u_{\Lambda}\in S^{1}_{0}(\Omega),$$ 
	such that
	(up to a subsequence and as $k\to+\infty$)
	\begin{itemize}
		\item[a)] $u_k\to u_{\Lambda}$ weakly in $S^{1}_{0}(\Omega)$ and strongly
		in $L^p(\Omega)$ for $1\leq p <2^{\star}_{Q}$;
		\item[b)] $u_k\to u_{\Lambda}$ a.e.\,in $\Omega$.
	\end{itemize}
	We now observe that, being $\{\lambda_k\}_k$ increasing, it follows that $\lambda_k\geq \lambda_1$ for every $k\geq 1$.
	Moreover, arguing as above yields that
	$u_{\lambda_k}\geq w_{\lambda_1}$, and thus
	$$u_{\Lambda} > 0\quad\text{a.e.\,in $\Omega$}.$$
	Moreover, since $u_k$ solves problem \eqref{EqProblem} (with $\lambda = \lambda_k$),
	we have
	$$\int_{\Omega}\langle \nabla_{\mathbb{G}}u_k,\nabla_{\mathbb{G}}\varphi\rangle_{\mathfrak{g}_{1}} -\lambda_k\int_\Omega u_k^{q}\varphi 
	-\int_\Omega u_k^{2^{\star}_{Q}-1}\varphi = 0\quad\text{for every $\varphi\in S^{1}_{0}(\Omega)$}.$$
Therefore, passing to the limit as $k\to+\infty$ in the above identity, and by dominated convergence, we get that $u_{\Lambda}$ satisfies
	$$\int_{\Omega}\langle \nabla_{\mathbb{G}}u_{\Lambda},\nabla_{\mathbb{G}}\varphi\rangle_{\mathfrak{g}_{1}} -\Lambda\int_\Omega u_{\Lambda}^{q}\varphi
	-\int_\Omega u_{\Lambda}^{2^{\star}_{Q}-1}\varphi = 0\quad\text{for every $\varphi\in S^{1}_{0}(\Omega)$},$$
	\noindent which shows that $u_{\Lambda}$ is actually a weak solution of
	problem \eqref{EqProblem} (with $\lambda = \Lambda$). 
%
	This closes the proof.
\end{proof}

\medskip

Before tackling the problem of a second solution, we focus
on the behavior of the functional $I_\lambda$ around $u_\lambda$. In particular we will show that the first solution $u_{\lambda}$ is actually a local minimum in the $S_{0}^{1}(\Omega)$-topology. As recalled in the Introduction, in the Euclidean case this is performed exploiting a famous result by Brezis and Nirenberg \cite{BNH1C1}. In our case we have to follow a different approach due to the lack of boundary regularity of the solution. In particular, we adapt the strategy used in \cite{AbDiVa} which in turn is inspired by a work of Alama.

\begin{lemma}\label{lem:locmin}
	For $\lambda\in (0,\Lambda)$, if $u_\lambda$ is the solution
	presented in Theorem \ref{thm:Existence_First}, then
	$u_\lambda$ is a local minimum for $I_\lambda$ in the $S^1_0(\Omega)$-topology,
	meaning that there exists $r_0>0$ such that for any $u\in S^1_0(\Omega)$,
	\[
	I_\lambda(u_\lambda)\leq I_\lambda(u) \quad \textrm{ for all } u\in S^{1}_{0}(\Omega) \textrm{ with } \|u-u_\lambda\|_{S_{0}^{1}(\Omega)}< r_0 .
	\]
\end{lemma}
\begin{proof}

	In the following of this proof, we denote by $\overline u$ the solution
	of \eqref{EqProblem} (with $\lambda = \overline \lambda$) for some value $\overline \lambda$
	such that  $\lambda<\overline \lambda <\Lambda$.
	
%
%
	
	Let's suppose by contradiction that there exists a sequence $\{v_n\}$ in $S^1_0(\Omega)$ verifying that
	$\|v_n-u_\lambda\|_{S_{0}^{1}(\Omega)}\to0$ and $I_\lambda(v_n)<I_\lambda(u_\lambda)$ for each $n$.
	We also introduce two ausiliary functions,
	\begin{align*}
		w_n&:=(v_n-\overline u)_+\, ,\\
		u_n&:=\max(0,\min(v_n,\overline u)),
	\end{align*}
	and the sets 
	\begin{align*}
		T_n&:=\{x\in \Omega:\ u_n(x)=v_n(x)\},\\
		S_n&:=\supp(w_n)\cap\Omega.
	\end{align*}
	
	We want to prove that 
	\begin{equation}\label{Sn0}
		\lim_{n\to+\infty}|S_n|=0.
	\end{equation}
	Consider the two sets,
	\begin{align*}
		E(n,\delta)&:=\{x\in \Omega:\ v_n(x)\geq\overline u(x)>u_\lambda(x)+\delta\},\\
		F(n,\delta)&:=\{x\in \Omega:\ v_n(x)\geq\overline u(x),\ \m{and}\ \overline u(x)\leq u_\lambda(x)+\delta\}.
	\end{align*}
	By construction, $S_n\subset E_n(n,\delta)\cup F(n,\delta)$ for each $n$ and each $\delta>0$.
	We are going to show that
	for any $\varepsilon>0$ and a suitable choice of $\delta>0$ and $n\in \N$, we have both $|E(n,\delta)|,|F(n,\delta)|<\frac{\varepsilon}{2}$.
	\begin{enumerate}
		\item We start from $E(n,\delta)$.
		By definition of the sequence $v_n$, we have $\|v_n-u_\lambda\|_{L^2(\Omega)}\to0$.
		Therefore if we fix $\delta$ and $\varepsilon$ there exists $n_0$ such that for any $n\geq n_0$,
		$\frac{\delta^2\varepsilon}{2}>\|v_n-u_\lambda\|_{L^2}$. As a consequence,
		\[
		\frac{\delta^2\varepsilon}{2}>\int_\Omega |v_n-u_\lambda|^2\geq\int_{E(n,\delta)}|v_n-u_\lambda|^2>\delta^2\cdot|E(n,\delta)|.
		\]
		This implies that $|E(n,\delta)|<\frac{\varepsilon}{2}$ for any $\delta$ and any $n\geq n_0$.
		
		\item Let's consider $F(n,\delta)$. If,
		\begin{align*}
			\overline F(\delta)&:=\{x\in \Omega:\ \overline u(x)\leq u_\lambda(x)+\delta\}\\
			\overline F&:=\{x\in \Omega:\ \overline u(x)\leq u_\lambda(x)\},
		\end{align*}
		then by construction
		\[
		0=|\overline F|=\left|\bigcap_{m=1}^\infty \overline F\left(\frac{1}{m}\right)\right|=\lim_{m\to+\infty}\left|\overline F\left(\frac{1}{m}\right)\right|.
		\]
		Therefore, for a suitable $m_0$ we get that $|\overline F(\delta)|<\frac{\varepsilon}{2}$ for any $\delta<\frac{1}{m_0}$
		and \cor{a fortiori} $|F(n,\delta)|<\frac{\varepsilon}{2}$ for any $n$ because $F(n,\delta)\subset\overline F(\delta)$.\newline
	\end{enumerate}
	
	This proves \eqref{Sn0}. Let's now consider the function
	\[
	h(u):=\frac{\lambda}{q+1}u_+^{q+1}+\frac{u_+^{2^\star_Q}}{2^\star_Q},
	\]
	and observe that by definition $\Omega=T_n\cup S_n$ because $u_n\leq v_n$.
	We develop the evaluations,
	\begin{align*}
		I_\lambda(v_n)&=\frac{1}{2}\|v_n\|_{S_{0}^{1}(\Omega)}^2-\int_\Omega h(v_n)\\
		&\geq \frac{1}{2}\|v_n^+\|_{S_{0}^{1}(\Omega)}^2+\frac{1}{2}\|v_n^-\|_{S_{0}^{1}(\Omega)}^2-\int_{T_n} h(v_n)-\int_{S_n}h(v_n)\\
		&(\m{because }u_n=\overline u\m{ and }v_n=w_n+\overline u\m{ on }S_n)\\
		&=\frac{1}{2}\|v_n^+\|_{S_{0}^{1}(\Omega)}^2+\frac{1}{2}\|v_n^-\|_{S_{0}^{1}(\Omega)}^2-\int_\Omega h(u_n)-\int_{S_n}(h(w_n+\overline u)-h(\overline u))\\
		&=I_\lambda(u_n)+\frac{1}{2}\left(\|v_n^+\|_{S_{0}^{1}(\Omega)}^2-\|u_n\|_{S_{0}^{1}(\Omega)}^2\right)+\frac{1}{2}\|v_n^-\|_{S_{0}^{1}(\Omega)}^2-\int_{S_n}\left(h(w_n+\overline u)-h(\overline u)\right).
	\end{align*}
	Here as $v_n^+=u_n+w_m$, using that $I_\lambda(u_n)\geq I_\lambda (u_\lambda)$
	because $u_n\in M$ for any $n$, and that $\overline u$ is a supersolution of \eqref{EqProblem},
	we get 
	\begin{equation}\label{eqADV}
		\begin{aligned}
		I_\lambda(v_n)&\geq I_\lambda (u_\lambda)+\frac{1}{2}\|w_n\|_{S_{0}^{1}(\Omega)}^2+\langle u_n,w_n\rangle_{S_{0}^{1}(\Omega)}+\frac{1}{2}\|v_n^-\|_{S_{0}^{1}(\Omega)}^2
		-\int_{S_n}\left(h(w_n+\overline u)-h(\overline u)\right)\\
		&\geq I_\lambda(u_\lambda)+\frac{1}{2}\|w_n\|_{S_{0}^{1}(\Omega)}^2+\frac{1}{2}\|v_n^-\|_{S_{0}^{1}(\Omega)}^2
		-\int_{S_n}\left(h(w_n+\overline u)-h(\overline u)-\lambda \overline u^qw_n-\overline u^{2^\star_Q-1}w_m\right).
		\end{aligned}
	\end{equation}
	
	Let's use the notation
	\[
	f_r(x,y)=\frac{1}{r+1}(x+y)^{r+1}-\frac{1}{r+1}y^{r+1}-xy^r.
	\]
	As $w_n\geq0$ and $\overline u\geq0$, simply
	by Taylor expasion we have
	\begin{equation}\label{tayq}
		0\leq f_{q}(w_n,\overline u)\leq \frac{q}{2}w_n^2\overline u^{q-1}.
	\end{equation}
		Mooreover, by \cite[Theorem 3.4]{Ruzhansky_Suragan}, we have 
	\begin{equation}\label{picone}
		 \lambda\int_\Omega \overline u^{q-1}w_n^2\leq \int_\Omega (-\Delta \overline u)\frac{w_n^2}{\overline u}\leq \|w_n\|_{S_{0}^{1}(\Omega)}^2,
	\end{equation}
	and combining the last two equations, we get
	\begin{equation}\label{combq}
		\lambda\int_\Omega f_q(w_n,\overline u)\leq \frac{q}{2}\int_\Omega w_n^2\overline u^{q-1}\leq \frac{q}{2}\|w_n\|_{S_{0}^{1}(\Omega)}^2.
	\end{equation}
 	By similar reasoning and using Sobolev inequality, we obtain
 \begin{equation}
 	\int_\Omega f_{2^\star_Q-1}(w_n,\overline u)\leq o(1)\cdot \|w_n\|_{S_{0}^{1}(\Omega)}^2.
 \end{equation}
 As $h(w_n+\overline u)=f_q(w_n,\overline u)+f_{2^\star_Q-1}(w_n,\overline u)$, 
 by combining the last results with \eqref{eqADV},
 we obtain
 \[
 I_\lambda(v_n)\leq I_\lambda(u_\lambda)+\frac{1}{2}\|w_n\|_{S_{0}^{1}(\Omega)}^2(1-q-o(1))+\frac{1}{2}\|v_n^-\|_{S_{0}^{1}(\Omega)}.
 \]
 As $I_\lambda(v_n)<I_\lambda(u_\lambda)$ by hypothesis,
 and $q<1$, then for $n$ sufficiently big we must have $v_n^-=0$,
 but this imply $v_n\geq0$ and therefore $v_n\in M$,
 thus contradicting $I_\lambda(v_n)<I_\lambda(u_\lambda)$.
 This completes the proof.	\end{proof}
\medskip

\color{black}

\medskip

\section{Proof of Theorem \ref{thm:main} - Part C)}\label{sec:Second_Solution}

The aim of this section is to prove that the problem \eqref{EqProblem} actually admits a second
solution for every $\lambda\in (0,\Lambda)$.
We briefly recall that, thanks to Theorem \ref{thm:Existence_First} and Lemma \ref{lem:locmin}, we have that 
\begin{itemize}
	\item for every $\lambda \in (0,\Lambda]$, the problem \eqref{EqProblem} admits a solution denoted by $u_{\lambda}$;
	\item for every $\lambda \in (0,\Lambda)$, $u_{\lambda}$ is a local minimum for $I_{\lambda}$ in the $S_{0}^{1}$-topology, i.e.
	there exists $r_0>0$ such that for any $u\in S^1_0(\Omega)$,
	\[
	I_\lambda(u_\lambda)\leq I_\lambda(u) \quad \textrm{ for all } u\in S^{1}_{0}(\Omega) \textrm{ with } \|u-u_\lambda\|_{S_{0}^{1}(\Omega)}< r_0 .
	\]
\end{itemize}

We now adapt to our setting the strategy used in \cite{Tarantello}.
Firstly, we are going to consider two cases:
\begin{enumerate}
	\item for every $r\in (0,r_0)$,
	\[
	\inf_{\|u-u_\lambda\|_{S_{0}^{1}(\Omega)}=r}I_\lambda(u)=I_\lambda(u_\lambda),
	\]
	\item there exists $r\in (0,r_0)$ such that 
	\[
	\inf_{\|u-u_\lambda\|_{S_{0}^{1}(\Omega)}=r}I_\lambda(u)>I_\lambda(u_\lambda).
	\]
\end{enumerate}
We treat separately the two cases in the following sub-sections. 
\newline

\subsection{First case}\label{subsec:CaseI}


We introduce the space
\begin{equation}\label{def:hlambda}
	H_\lambda:=\{u\in S^1_0(\Omega):\,u\geq u_\lambda \m{ a.e.~in }\Omega\}.
\end{equation}

\noindent By the standing assumptions in (1), there exists 
$\{u_k\}_k\subset H_{\lambda}$ such that 
\begin{enumerate}[i)]
	\item $\|u_k-u_\lambda\|_{S^1_0(\Omega)} = r$ for every $k\geq 1$;	
	\item as $k\to+\infty$ we have $I_\lambda(u_k)\to I_\lambda(u_\lambda)$.
\end{enumerate}
We also introduce the space
$$X_\lambda := \{u\in H_\lambda:\,r-\bar r\leq\|u-u_\lambda\|_{S^1_0(\Omega)}\leq r+\bar r\},$$
where $\bar r>0$ is taken sufficiently small to have $r-\bar r>0$ and $r+\bar r<r_0$.
The set $X_\lambda$ becomes a complete metric space once endowed with the distance associated to the norm $\|\cdot\|_{S^1_0(\Omega)}$.

We can now proceed very similarly to \cite{BiGaVe},
by applying the Ekeland's Variational Principle 
and therefore obtaining a sequence
 $\{w_k\}_k\subset X_\lambda$ such that
\begin{equation} \label{eq:EkelandCaseA}
	\begin{split}
		\mathrm{i)}&\,\, I_\lambda(w_k)\leq I_\lambda(u_k)\leq I_\lambda(u_\lambda)+\frac{1}{k^2}, \\
		\mathrm{ii)}&\,\,\|w_k-u_k\|_{S^1_0(\Omega)}\leq \frac{1}{k}, \\
		\mathrm{iii)}&\,\, I_\lambda(w_k)\leq I_\lambda(u)+\frac{1}{k}\,
		\|w_k-u\|_{S^1_0(\Omega)}\quad\text{for every $u\in X_\lambda$}.    
	\end{split}
\end{equation}
By boundedness of $\{w_k\}$ in $S^1_0(\Omega)$, there exists
$w_\lambda\in S^1_0(\Omega)$ such that 
the following are true (up to a sub-sequence),
\begin{equation} \label{eq:limitvkCaseA}
	\begin{split}
		\mathrm{i)}&\,\,\text{$w_k\to w_\lambda$ weakly in $S^1_0(\Omega)$}; \\
		\mathrm{ii)}&\,\,\text{$w_k\to w_\lambda$ strongly in $L^r(\Omega)$ for every $1\leq r<2^\star_Q$}; \\
		\mathrm{iii)}&\,\,\text{$w_k\to w_\lambda$ pointwise
			a.e.\,in $\Omega$}.
	\end{split}
\end{equation}

First we show that the limit function $w_{\lambda}$ is a solution to \eqref{EqProblem}. This is the content of the following

\begin{lemma}\label{lem:casoa1}
	The function $w_\lambda$ is a weak solution of \eqref{EqProblem}.
\end{lemma}
\proof
Given $w\in H_\lambda$ we consider $\varepsilon_0$
sufficiently small that $w_k+\varepsilon(w-w_k)\in X_\lambda$ for each $0<\varepsilon<\varepsilon_0$.
For $k$ sufficiently big such an $\varepsilon_0$ always exists 
because, indeed
\begin{multline*}
 r-\frac{1}{k}\leq \|u_k-u_\lambda\|_{S^1_0(\Omega)}-\|w_k-u_k\|_{S^1_0(\Omega)}\leq \|w_k-u_\lambda\|_{S^1_0(\Omega)}\\
 \leq \|u_k-u_\lambda\|_{S^1_0(\Omega)}+\|w_k-u_k\|_{S^1_0(\Omega)}\leq r+\frac{1}{k}.
\end{multline*}
By setting $u=w_k+\varepsilon(w-w_k)$ in \eqref{eq:EkelandCaseA}
we get
\[
\frac{I_\lambda(w_k+\varepsilon(w-w_k))-I_\lambda(w_k)}{\varepsilon}\geq -\frac{1}{k}\, \|w-w_k\|_{S^1_0(\Omega)}.
\]
Letting $\varepsilon\to0^+$, we get
\begin{multline}\label{eq:ineq2.21haitao}
		 -\frac{1}{k}\, \|w-w_k\|_{S^1_0(\Omega)}\leq \int_\Omega\langle \nabla_\G w_k,\nabla_\G(w-w_k)\rangle_{\mathfrak{g}_{1}}\\
	-\int_\Omega w_k^{2^\star_Q-1}(w-w_k)
		-\lambda \int_\Omega w_k^q(w-w_k)\quad\textrm{ for every } w\,\,\in H_\lambda.
\end{multline}
For any $\varphi\in S^1_0(\Omega)$ and $\varepsilon>0$ we introduce the functions
$$\varphi_{k,\varepsilon}:=w_k+\varepsilon\varphi-u_\lambda \quad \textrm{ and } \quad 
\varphi_{\varepsilon}:=w_\lambda+\varepsilon\varphi-u_\lambda.$$
We further set $w:=w_k+\varepsilon\varphi+(\varphi_{k,\varepsilon})_-\in H_\lambda$, 
then, by \eqref{eq:ineq2.21haitao},
\begin{multline}\label{eq:ineq2.22haitao}
-\frac{1}{k}\, \|\varepsilon\varphi+(\varphi_{k,\varepsilon})_-\|_{S^1_0(\Omega)}\leq \int_\Omega\langle \nabla_\G w_k,\nabla_\G(\varepsilon\varphi+(\varphi_{k,\varepsilon})_-)\rangle_{\mathfrak{g}_{1}}\\
-\int_\Omega w_k^{2^\star_Q-1}(\varepsilon\varphi+(\varphi_{k,\varepsilon})_-)
-\lambda\int_\Omega w_k^{q}(\varepsilon\varphi+(\varphi_{k,\varepsilon})_-).
\end{multline}
From \eqref{eq:limitvkCaseA},
$$(\varphi_{\varepsilon,k})_-\to(\varphi_\varepsilon)_-  \quad \textrm{ a.e. in } \Omega, \textrm{ as } k\to+\infty.$$
Moreover, we have
\[
w_k^{2^\star_Q-1}(\varphi_{k,\varepsilon})_-=w_k^{2^\star_Q-1}(u_\lambda-\varepsilon\varphi-w_k)\cdot \mathbf{1}_{\{u_\lambda-\varepsilon\varphi-w_k\}}
\leq (u_\lambda+\varepsilon|\varphi|)^{2^\star_Q}.
\]
Therefore, by Dominated Convergence we get
\begin{equation}\label{eq:gather1}
	\begin{split}
	\lim_{k\to+\infty}\biggl(\int_\Omega w_k^{2^\star_Q-1}(\varepsilon\varphi+(\varphi_{k,\varepsilon})_-)
	&+\lambda\int_\Omega w_k^{q}(\varepsilon\varphi+(\varphi_{k,\varepsilon})_-)\biggl)\\
	&=\int_\Omega w_\lambda^{2^\star_Q-1}(\varepsilon\varphi+(\varphi_\varepsilon)_-)+
	\lambda\int_\Omega w_\lambda^{q}(\varepsilon\varphi+(\varphi_\varepsilon)_-)
	\end{split}
\end{equation}
For the other term, similarly to  \cite[Lemma 3.4]{BadTar}
we have
\[
\int_\Omega \langle \nabla_\G w_k,\nabla_\G (\varphi_{k,\varepsilon})_-\rangle_{\mathfrak{g}_{1}}\leq
\int_\Omega\langle \nabla_\G{w_\lambda},\nabla_\G(\varphi_{\varepsilon})_-\rangle_{\mathfrak{g}_{1}} +o(1)\ \m{as}\ k\to+\infty
\]
and because $v_k\rightharpoonup v_\lambda$  weakly in $S_0^1(\Omega)$, we obtain
\begin{equation}\label{eq:gather2}
	\int_\Omega\langle \nabla_\G w_k,\nabla_\G(\varepsilon\varphi+(\varphi_{k,\varepsilon})_-)\rangle_{\mathfrak{g}_{1}}
	\leq \int_\Omega\langle \nabla_\G {w_\lambda},\nabla_\G(\varepsilon\varphi+(\varphi_{\varepsilon})_-)\rangle_{\mathfrak{g}_{1}} +o(1)\ \m{as}\ k\to+\infty.
\end{equation}
As $\|w_k\|_{S^1_0(\Omega)}$ is uniformly bounded
w.r.t. $k$, we have the same
for $\|(\varphi_{k,\varepsilon})_-\|_{S^1_0(\Omega)}$.
Therefore we can pass to the limit as $k\to+\infty$
 in \eqref{eq:ineq2.22haitao}, 
and by recalling \eqref{eq:gather1} and \eqref{eq:gather2}  we obtain
\begin{equation}
	\int_\Omega\langle \nabla_\G w_\lambda,\nabla_\G(\varepsilon\varphi+(\varphi_\varepsilon)_-)\rangle_{\mathfrak{g}_1}\geq
	\int_\Omega w_\lambda^{2^\star_Q-1}(\varepsilon\varphi+(\varphi_\varepsilon)_-)+
	\lambda\int_\Omega w_\lambda^{q}(\varepsilon\varphi+(\varphi_\varepsilon)_-).
\end{equation}
Finally, we can conlude as in \cite[Lemma 4.1]{BiGaVe}, getting
\begin{equation}
	\int_\Omega\langle \nabla_\G w_\lambda,\nabla_\G\varphi\rangle_{\mathfrak{g}_{1}}-\lambda\int_\Omega w_\lambda^{q}\varphi
	-\int_\Omega v_\lambda^{2^\star_Q-1}\varphi\geq0,
\end{equation}
and by the arbitrariness of $\varphi \in S^1_0(\Omega)$ 
we conclude that $w_\lambda$
is a weak solution of \eqref{EqProblem}.\fine

\color{black}

The following technical Lemma closely follows \cite[Lemma 4.2]{BiGaVe}.

\begin{lemma}\label{lem:casoa2}
	If $w_\lambda$ is as above,
	then $\|w_\lambda-u_\lambda\|_{S^{1}_{0}(\Omega)}=r$.
\end{lemma}
\begin{proof} 
We notice first that it is enough to prove that 
\begin{equation}\label{eq:stronconv}
	w_k\to w_\lambda\ \m{strongly in }S_0^1(\Omega)\ \m{as}\ k\to+\infty.
\end{equation}
Indeed, by using that $\|u_k-u_\lambda\|_{S^{1}_{0}(\Omega)}=r$ for any $k$,
we get
\[
r-\|w_k-u_k\|_{S^{1}_{0}(\Omega)} \leq \|w_k-u_\lambda\|_{S^{1}_{0}(\Omega)} \leq r+\|w_k-u_k\|_{S^{1}_{0}(\Omega)}.
\]
Combining the latter with both the strong convergence in $S_{0}^{1}(\Omega)$ and \eqref{eq:EkelandCaseA}-ii),
implies 
$$\|w_\lambda-u_\lambda\|_{S^{1}_{0}(\Omega)}=r.$$
Let us now proceed with the proof of \eqref{eq:stronconv}. 
In view of \eqref{eq:limitvkCaseA}, and arguing as in the proof of \cite[Lemma 4.3]{BiGaVe}, it follows that

%
\begin{align}
	\|w_k - w_\lambda\|_{L^{q+1}(\Omega)} &\to 0 \quad \textrm{ as } k \to +\infty, \label{eq:case1concl1}\\
	\|w_k\|^{2^\star_Q}_{L^{2^\star_Q}}(\Omega)&=\|w_\lambda\|_{L^{2^\star_Q}(\Omega)}^{2^\star_Q}+
	\|w_k-w_\lambda\|_{L^{2^\star_Q}(\Omega)}^{2^\star_Q}+o(1),\\
	\|w_k\|_{S^{1}_{0}(\Omega)}^2&=\|w_\lambda\|_{S^{1}_{0}(\Omega)}^2+\|w_k-w_\lambda\|_{S^{1}_{0}(\Omega)}^2+o(1).\label{eq:case1concl2}
\end{align}
In particular, from \eqref{eq:case1concl1} we get
\begin{equation}\label{eq:case1concl1bis}
	\int_\Omega w_k^{q+1}=\int_\Omega w_\lambda^{q+1}+o(1), \quad \textrm{ as } k \to +\infty. 
\end{equation}
%
%
Therefore, choosing $w=w_\lambda \in H_\lambda$ in \eqref{eq:ineq2.21haitao}, we obtain
\begin{equation}
	\begin{aligned}
		\|w_k - w_{\lambda}\|_{S^{1}_{0}(\Omega)}^2 &= -\int_{\Omega} \langle \nabla_{\mathbb{G}}w_k, \nabla_{\mathbb{G}}(w_{\lambda}-w_{k})\rangle_{\mathfrak{g}_1} + \int_{\Omega} \langle \nabla_{\mathbb{G}}w_{\lambda}, \nabla_{\mathbb{G}}(w_{\lambda}-w_{k})\rangle_{\mathfrak{g}_1}\\
		&\leq \dfrac{1}{k}\|w_{\lambda}-w_{k}\|_{S^{1}_{0}(\Omega)} + \int_{\Omega}w_{k}^{2\star_{Q}-1}(w_k - w_{\lambda}) + \lambda \int_{\Omega}w_{k}^{q}(w_k - w_{\lambda})\\
		&\qquad + \int_{\Omega} \langle \nabla_{\mathbb{G}}w_{\lambda}, \nabla_{\mathbb{G}}(w_{\lambda}-w_{k})\rangle_{\mathfrak{g}_1}\\
		&=\int_{\Omega}w_{k}^{2^{\star}_{Q}-1}(w_{\lambda}-w_k) + \lambda \int_{\Omega}w_{k}^{q+1} - \lambda \int_{\Omega}w_{k}^{q}w_{\lambda} +o(1)\\
		&=\|w_k-w_{\lambda}\|^{2^{\star}_{Q}}_{L^{2^{\star}_{Q}-1}(\Omega)} + \|w_{\lambda}\|^{2^{\star}_{Q}}_{L^{2^{\star}_{Q}-1}(\Omega)}-\int_{\Omega}w_{k}^{2^{\star}_{Q}-1}w_{\lambda} \\
		&\qquad + \lambda \int_{\Omega}w_{k}^{q+1} - \lambda \int_{\Omega}w_{k}^{q}w_{\lambda} +o(1)\\
		&=\|w_k-w_{\lambda}\|^{2^{\star}_{Q}}_{L^{2^{\star}_{Q}-1}(\Omega)} + o(1) \quad \textrm{ as } k \to +\infty.
	\end{aligned}
\end{equation}

\noindent To proceed further, we choose $w=2w_k\in H_\lambda$,
yielding
\begin{equation}\label{eq:usando_2wk}
\|w_k\|_{S^{1}_{0}(\Omega)}^2-\|w_k\|_{L^{2^\star_Q}(\Omega)}^{2^\star_Q}-\lambda\int_\Omega w_k^{q+1}\geq-\frac{1}{k}\|w_k\|_{S^{1}_{0}(\Omega)}^2=o(1).
\end{equation}
Since $w_{\lambda}$ is actually a solution of \eqref{EqProblem}, we get
\begin{equation}\label{eq:wlambda_is_sol}
\|w_\lambda\|_{S^{1}_{0}(\Omega)}^2-\|w_\lambda\|_{L^{2^\star_Q}(\Omega)}^{2^\star_Q}-\lambda\int_\Omega w_k^{1-\gamma}\geq o(1)\ \ \m{as}\ k\to+\infty.
\end{equation}
Now, combining \eqref{eq:usando_2wk} with \eqref{eq:wlambda_is_sol} yields
 \begin{equation}\label{eq:vlambda1}
 	\|w_k-w_\lambda\|_{S^{1}_{0}(\Omega)}^2\geq \|w_k-w_\lambda\|_{L^{2^\star_Q}(\Omega)}^{2^\star_Q}+o(1)\ \ \m{as}\ k\to+\infty.
 \end{equation}
 Assuming without loss of generality that $I_\lambda(u_\lambda)\leq I_\lambda(w_\lambda)$,
 from \eqref{eq:EkelandCaseA} and \eqref{eq:case1concl1}-\eqref{eq:case1concl2}
 we obtain
 \[
 \begin{split}
 	I_\lambda(w_k-w_\lambda)&=I_\lambda(w_k)-I_\lambda(w_\lambda)+o(1)\\
 	&\leq I_\lambda(u_\lambda)-I_\lambda(w_\lambda)+\frac{1}{k^2}+o(1)\\
 	&=o(1)\ \ \m{as}\ k\to+\infty,
 \end{split}
 \]
 which, together with \eqref{eq:case1concl1}, gives
 \begin{equation}\label{eq:vlambda2}
 	\frac{1}{2}\|w_k-w_\lambda\|_{S^{1}_{0}(\Omega)}^2-\frac{1}{2^\star_Q}\|w_k-w_\lambda\|^{2^\star_Q}_{L^{2^\star_Q}(\Omega)}=
 	I_\lambda(w_k-w_\lambda)+\frac{\lambda}{q+1}\int_\Omega|w_k-w_\lambda|^{q+1}\leq o(1).
 \end{equation}
 From \eqref{eq:vlambda1} and \eqref{eq:vlambda2}
 we finally conclude
 \[
 \lim_{k\to+\infty}\|w_k-w_\lambda\|_{L^{2^\star_Q}(\Omega)}^{2^\star_Q}=\lim_{k\to+\infty}\|w_k-w_\lambda\|_{S^{1}_{0}(\Omega)}^2=0,
 \]
 proving \eqref{eq:stronconv}. This closes the proof.
\end{proof}

Combining Lemma \ref{lem:casoa1} with Lemma \ref{lem:casoa2} yields the existence of a second 
solution $w_\lambda$ to \eqref{EqProblem} for every $\lambda \in (0,\Lambda)$ in the case that
for every $r\in (0,r_0)$,
\[
\inf_{\|u-u_\lambda\|_{S_{0}^{1}(\Omega)}=r}I_\lambda(u)=I_\lambda(u_\lambda).
\]

\subsection{Second case} \label{subsec:CaseII}


%


As before, we consider the space $H_\lambda$ introduced in \eqref{def:hlambda} and
$C([0,1],H_\lambda)$ the space of continuous curves endowed with the following distance,
\begin{equation}\label{maxdistance}
d(\eta,\eta')=\max_{t\in[0,1]}\|\eta(t)-\eta'(t)\|_{S^1_0(\Omega)}.
\end{equation}
Moreover, we consider the following subspace
\begin{equation}\label{def:gammal}
\Gamma_\lambda:=\left\{\eta\in C([0,1],H_\lambda):\begin{array}{l}
	 \eta(0)=u_\lambda,\\[0.1cm] 
	\|\eta(1)-u_\lambda\|_{S^{1}_{0}(\Omega)}>r_1,\\[0.1cm]
	I_\lambda(\eta(1))<I_\lambda(u_\lambda)
\end{array}
\right\}.
\end{equation}

\noindent We first show that $\Gamma_\lambda\neq \varnothing$
and then we provide an estimate of the minimax level
\[
\ell_0:=\inf_{\eta\in\Gamma_\lambda}\max_{t\in[0,1]}I_\lambda(\eta(t)).
\]

\noindent Once this is done, we will apply once again the Ekeland's variational principle.

\noindent We consider the family of functions $U_\varepsilon$ defined
at \eqref{eq:def_u_varepsilon}.
For any $a\in \G$, we take
\[
U_{\varepsilon,a}(g):=U_\varepsilon(a^{-1}\diamond g)
= {\varphi(a^{-1}\diamond g) T_{\varepsilon}(a^{-1}\diamond g)}.
\]
We recall that $\{T_\varepsilon\}$ is a family
of functions defined in \eqref{eq:Minimi_riscalati} such that  $T_1=T$
is a minimizer of the Sobolev Inequality \eqref{eq:Sobolev_Ineq}.

\begin{lemma}\label{lem:claim2.7haitao}
	There exists $\varepsilon_0>0$, $a\in\Omega$ and $R_0\geq1$ such that
\begin{align}
	 & I_\lambda(u_\lambda + R U_{\varepsilon,a}) < I_\lambda(u_\lambda) 
	&& \forall \varepsilon \in (0, \varepsilon_0),\ \forall R \geq R_0, \\
& I_\lambda(u_\lambda + t R_0 U_{\varepsilon,a}) < I_\lambda(u_\lambda) + \frac{1}{Q} S_\G^{Q/2} 
	&& \forall t \in [0, 1],\ \forall \varepsilon \in (0, \varepsilon_0),\label{eq:lemma2.7Haitao}
\end{align}
where $S_\G$ is the best Sobolev constant defined as in  \eqref{eq:def_best_constant}.
\end{lemma}
\begin{proof}
Following the computations already developed in \cite{BiGaVe},
\begin{align*}
    I_\lambda(u+tRU_{\varepsilon,a})&=
    \frac{1}{2}\int_{\Omega}|\nabla_{\G}u|^2-\frac{1}{2^\star_Q}\int_{\Omega} |u|^{2^\star_Q}-\frac{\lambda}{q+1}\int_{\Omega} |u|^{q+1}\\
    &\ \ \ \ +tR \left(\int_{\Omega}\langle \nabla_{\G}u,\nabla_{\G}U_{\varepsilon,a}\rangle_{\mathfrak{g}_{1}}-\int_{\Omega} u^{2^\star_Q -1}U_{\varepsilon,a}-\lambda\int u^qU_{\varepsilon,a}\right)\tag{C}\label{eq:C}\\
    &\ \ \ \ -\frac{1}{2^\star_Q}\left(\int|u+tRU_{\varepsilon,a}|^{2^\star_Q}-\int_{\Omega} |u|^{2^\star_Q}\right)+tR\, \int_{\Omega} u^{2^\star_Q -1}U_{\varepsilon,a}\tag{A}\label{eq:A1}\\
     &\ \ \ \ -\frac{\lambda}{q+1}\left(\int_{\Omega}|u+tRU_{\varepsilon,a}|^{q+1}-\int_{\Omega} |u|^{q+1}\right)+ \lambda tR\, \int_{\Omega} u^qU_{\varepsilon,a}\tag{A2}\label{eq:A2}\\
    &\ \ \ \ +\frac{t^2R^2}{2} \int_{\Omega}|\nabla_{\G}U_{\varepsilon,a}|^2.\tag{B}\label{eq:B}
\end{align*}
Taking the limit for $\varepsilon\to0^+$, we 
start by estimating~\eqref{eq:A2}.
We observe that 
\[
tRu^{2^\star_Q -1}U_{\varepsilon,a}=\left.\left((u+s)^{2^\star_Q}\right)'\right|_{s=tRU_{\varepsilon,a}}.
\]
Therefore, by  strict convexity the term \eqref{eq:A2} must be negative.

For the other terms, we take analogous estimates to the ones developed
in \cite{BiGaVe}. In particular, \eqref{eq:C} is null by definition of
weak solution with $U_{\varepsilon,a}$ as test function,
while the other terms have been estimated in \cite[Lemma 4.3]{BiGaVe}.
We resume,
\begin{align*}
    \eqref{eq:C}&\ =0,\\
\eqref{eq:A1}&\ =-\frac{A\, t^{2^\star_Q}R^{2^\star_Q}}{2^\star_Q}-t^{2^\star_Q -1}R^{2^\star_Q -1}\, K\, \varepsilon^{\frac{Q-2}{2}}+o\left(\varepsilon^{\frac{Q-2}{2}}\right),\\
    \eqref{eq:B}&\ =\frac{B\, t^2R^2}{2}+o\left(\varepsilon^{\frac{Q-2}{2}}\right),
\end{align*}
for suitable positive constants $A,B,K$.
In particular,
\begin{align*}
	A&=\|T\|_{L^{2^\star_Q}(\G)},\\
	B&=\||\nabla_\G T|\|_{L^2(\G)}^2,
\end{align*}
and both can be estimated with the Sobolev constant
as done in Lemma \ref{lem:Stime_Talentiane_modificate}.
Moreover, $K$ is a constant depending on $a$ with the property that
\[
\int_\Omega uU_{\varepsilon,a}^{2^\star_Q -1}=K\varepsilon^{(Q-2)\slash 2}+o(\varepsilon^{(Q-2)\slash2}).
\]

To conclude the proof of the lemma 
we follow the approach of \cite{Tarantello}.
We slightly change the notation by posing
 $s:=tR$ and $S:=\left(\frac{B}{A}\right)^{\frac{1}{2^\star_Q -2}}$,
and, we introduce
\[
f_\varepsilon(s):=\frac{B\, s^2}{2}-\frac{A \, s^{2^\star_Q}}{2^\star_Q}-s^{2^\star_Q -1}\, K\, \varepsilon^n.
\]
As \eqref{eq:A2} is negative, we have
\begin{equation}\label{disIf}
	I_\lambda(u+tRU_{\varepsilon,a})<f_\varepsilon(tR)+o\left(\varepsilon^{\frac{Q-2}{2}}\right).
\end{equation}

With the same calculations developed in \cite{BiGaVe}
we get
%
%
\[
I_\lambda(u+tRU_{\varepsilon,a})<I_\lambda(u)+\left(\frac{1}{2}-\frac{1}{2^\star_Q}\right)\cdot \frac{B^{\frac{2^\star_Q}{2^\star_Q -2}}}{A_1^{\frac{2}{2^\star_Q -2}}}-S^{2^\star_Q -1}K\varepsilon^{\frac{Q-2}{2}}
+o\left(\varepsilon^{\frac{Q-2}{2}}\right),
\]
and this allows to conclude because $\frac{1}{2}-\frac{1}{2^\star_Q}=\frac{1}{Q}$
and $\frac{B}{A_1^\frac{2}{2^\star_Q}}=S_{\G}$. This closes the proof.
\end{proof}

We notice that by Lemma \ref{lem:claim2.7haitao}, $\Gamma_\lambda\neq \varnothing$ and moreover $\Gamma_\lambda$ is a complete metric space endowed
with the distance~\eqref{maxdistance}. This follows simply by the fact that
\begin{equation}\label{eq:eta}
	\widetilde \eta(t) := u_\lambda+tR_0U_{\varepsilon,a}\in \Gamma_\lambda\quad\text{for all $\varepsilon\in (0,\varepsilon_0)$},
\end{equation}
\noindent eventually enlarging $R_0$.
We now proceed by showing that problem \eqref{EqProblem} do actually admit a second solution.

\begin{teo}\label{thm:casob}
	For $0<\lambda<\Lambda$, if there exists $r\in (0,r_0)$ such that 
	\[
	\inf_{\|u-u_\lambda\|_{S_{0}^{1}(\Omega)}=r}I_\lambda(u)>I_\lambda(u_\lambda),
	\]
	then there exists a solution $v_\lambda$ to \eqref{EqProblem}
	such that $v_\lambda \not\equiv u_\lambda$.\newline
\end{teo}
\begin{proof}
We start by considering the Generalized Directional Derivative
introduced in \cite{BadTar}, 
\[
I^0_\lambda(w;v):=\limsup_{\|h\|\to0^+,\, \rho\to0^+}\frac{I_\lambda(w+h+\rho v)-I_\lambda(w+h)}{\rho}.
\]
Some basic properties of this object are proved in \cite[Appendix]{BiGaVe}.
In the case of object,
\[
I_\lambda^0(w;v)=\int_\Omega\langle \nabla_\G w,\nabla_\G v\rangle_{\mathfrak{g}_1}-\lambda\int_\Omega w^{q}v-\int_\Omega w^{2^\star_Q-1}v.
\]
We recall the functional
\[
\Phi(\eta):= \max_{t\in [0,1]}I_\lambda(\eta(t)).
\]
and its minimax level
\[
\ell_0:=\inf_{\eta\in\Gamma_\lambda}\Phi(\eta).
\]
By Ekeland's Variational Principle, there exists
a sequence $\{\eta_k\}_k\in \Gamma_\lambda$
such that
$$\Phi(\eta_k)\leq \ell_0+\frac{1}{k} \quad \textrm{ and } \quad \Phi(\eta_k)\leq \Phi(\eta)+\frac{1}{k}d(\eta_k,\eta).$$
We now recall the result of \cite[Lemma A.2]{BiGaVe}.
For every $k\in \mathbb{N}$, we define 
\[
\Lambda_k:=\left\{t\in(0,1):\ I_\lambda(\eta_k(t))=\max_{s \in [0,1]}I_\lambda(\eta_k(s))\right\}.
\]
Then, for every $k$ there exists $t_k\in \Lambda_k$ such that 
for $v_k:=\eta_k(t_k)$,
\begin{equation}\label{lemmaimp}
	I^0_\lambda(v_k;w-v_k)\geq-\frac{\max(1,\|w-v_k\|)}{k}\quad\textrm{ for every } w\in H_\lambda.
\end{equation}
Resuming what we just presented,
\begin{itemize}
	\item $I_\lambda(v_k)\to\ell_0$ as $k\to+\infty$,
	\item there exists $C>0$ such that $\forall w\in H_\lambda$,
	\begin{equation}\label{eq:badtarfin}
		\int_\Omega\langle \nabla_\G v_k,\nabla_\G(w-v_k)\rangle_{\mathfrak{g}_{1}} -\lambda\int_\Omega v_k^{q}(w-v_k)-\int_\Omega v_k^{2^\star_Q-1}(w-v_k)\geq
		-\frac{C}{k}(1+\|w\|_{S^{1}_{0}(\Omega)}).
	\end{equation}
\end{itemize}
By choosing $w=2v_k$,
we get
\begin{equation}\label{eq:ultima}
\| v_k\|^2_{{S_0^1(\Omega)}}-\|v_k\|^{2^\star_Q}_{L^{2^\star_Q}(\Omega)}-\lambda \|v_k\|^{q+1}_{L^{q+1}(\Omega)}\geq-\frac{C}{k}\max(1,\|v_k\|_{{S_0^1(\Omega)}}).
\end{equation}

\noindent By \eqref{eq:ultima}, and recalling \eqref{eq:Def_Ilambda}, it follows that
\begin{equation*}
	\begin{aligned}
	\dfrac{1}{2}&\| v_k\|^2_{{S_0^1(\Omega)}} - \dfrac{\lambda}{q+1}\|v_k\|^{q+1}_{L^{q+1}(\Omega)} - \dfrac{1}{2^\star_Q}\|v_k\|^{2^\star_Q}_{L^{2^\star_Q}(\Omega)} \\
	&-\dfrac{1}{2^{\star}_{Q}}\left(\| v_k\|^2_{{S_0^1(\Omega)}}-\|v_k\|^{2^\star_Q}_{L^{2^\star_Q}(\Omega)}-\lambda \|v_k\|^{q+1}_{L^{q+1}(\Omega)} +\frac{C}{k}\max(1,\|v_k\|_{{S_0^1(\Omega)}}) \right) \leq I_{\lambda}(v_k).
	\end{aligned}
\end{equation*}

\noindent As said above, $I_\lambda(v_k)\to \ell_0$ for $k\to+\infty$,
and therefore we obtain

\begin{equation}\label{eq:gamma0}
	\ell_0+o(1)\geq \left(\frac{1}{2}-\frac{1}{2^\star_Q}\right)\|v_k\|_{{S_0^1(\Omega)}}^2-\lambda\left(\frac{1}{q+1}-\frac{1}{2^\star_Q}\right)\|v_k\|_{L^{q+1}(\Omega)}^{q+1}.
\end{equation}


Suppose $v_k$ is unbouded in $S^1_0(\Omega)$. As
$\frac{1}{2}-\frac{1}{2^\star_Q}>0$, 
then up to a subsequence we have
 $\|v_k\|_{S^{1}_{0}(\Omega)}\to+\infty$,
contradicting \eqref{eq:gamma0}.
Therefore $v_k$ is bounded
and \eqref{lemmaimp} with
$w=2v_k$ implies
\begin{equation}\label{eq:last}
	\|v_k\|_{S^{1}_{0}(\Omega)}^2-\lambda\int_\Omega v_k^{q+1}-\int_\Omega v_k^{2^\star_Q}\geq
	-\frac{C}{k}(1+2\|v_k\|_{S^{1}_{0}(\Omega)}).
\end{equation}

Now, we can proceed following almost verbatim \cite{BiGaVe}. Arguing as in 
the proofs of Lemmas \ref{lem:casoa1} and \ref{lem:casoa2},
we can prove both that
$v_k$ weakly converges (up to a subsequence) to a weak solution $v_\lambda$
of \eqref{EqProblem} and that
\begin{equation}\label{eq:convergence}
	\|v_k-v_\lambda\|_{S^{1}_{0}(\Omega)}^2-\|v_k-v_\lambda\|_{L^{2^\star_Q}(\Omega)}^{2^\star_Q}=o(1)\ \ \m{as}\ k\to+\infty.
\end{equation}

It only remains to show that $v_\lambda\not \equiv u_\lambda$, i.e. that we truly found a second solution of \eqref{EqProblem}.
To this aim, we notice first that for any $\eta\in \Gamma_\lambda$,
\[
\|\eta(0)-u_\lambda\|_{S^{1}_{0}(\Omega)}=0\ \ \m{and}\ \ \|\eta(1)-u_\lambda\|_{S^{1}_{0}(\Omega)}>r_1,
\]
therefore there exists $t_\eta\in [0,1]$
such that $\|\eta(t_\eta)-u_\lambda\|_{S^{1}_{0}(\Omega)}=r_1$.
Since we are dealing with (2),
we have
\begin{equation}\label{eq:Stima_l0_basso}
\ell_0=\inf_{\Gamma_\lambda}\Phi(\eta)\geq\inf_{\Gamma_\lambda}I_\lambda(\eta(t_\eta))
\geq\inf_{\|u-u_\lambda\|_{S^{1}_{0}(\Omega)}=r_1} I_\lambda(u)>I_\lambda(u_\lambda),
\end{equation}
\noindent where the last infimimum is taken among the $u\in H_\lambda$
such that $\|u-u_\lambda\|_{S^{1}_{0}(\Omega)}=r_1$. On the other hand
if we consider $\widetilde\eta$ defined in \eqref{eq:eta},
then by \eqref{eq:lemma2.7Haitao}
we get,
\begin{equation}\label{eq:Stima_l0_alto}
\ell_0\leq \Phi(\widetilde \eta)=\max_{t\in [0,1]}I_\lambda(\widetilde \eta(t))<
I_\lambda(u_\lambda)+\frac{1}{Q}S_\G^{Q\slash2}.
\end{equation}
Combining \eqref{eq:Stima_l0_basso} with \eqref{eq:Stima_l0_alto} we find
\begin{equation}\label{eq:sumup}
	I_\lambda(u_\lambda)<\ell_0<I_\lambda(u_\lambda)+\frac{1}{Q}S_\G^{Q\slash2}.
\end{equation}
We further stress that, since $v_k\to v_\lambda$ weakly in $S^1_0(\Omega)$, then
the equalities \eqref{eq:case1concl1}-\eqref{eq:case1concl2} still hold true.
Therefore, by \eqref{eq:sumup} and recalling that $I_\lambda(v_k)\to \ell_0$,
we get (for $k$ sufficiently large), that there exists a positive number $\delta_{0}$ such that 
\begin{equation}\label{eq:sumup2}
	\begin{split}
		\frac{1}{2}&\|v_k-v_\lambda\|_{S^{1}_{0}(\Omega)}^2-\frac{1}{2^\star_Q}\|v_k-v_\lambda\|_{L^{2^\star_Q}(\Omega)}^{2^\star_Q}\\
		&=\frac{1}{2}(\|v_k\|_{S^{1}_{0}(\Omega)}^2-\|v_\lambda\|_{S^{1}_{0}(\Omega)}^2)-\frac{1}{2^\star_Q}(\|v_k\|^{2^\star_Q}_{L^{2^\star_Q}(\Omega)})+o(1)\\
		&=I_\lambda(v_k)-I_\lambda(u_\lambda)+o(1)\\
		&=\ell_0-I_\lambda(u_\lambda)+o(1)\\
		&<\frac{1}{Q}S_\G^{Q\slash2}-\delta_0,
	\end{split}
\end{equation}
\noindent and the last term is positive.
From \eqref{eq:convergence}, \eqref{eq:sumup} and \eqref{eq:sumup2},
closely following \cite[Proposition 3.1]{Tarantello},
we get that $v_k\to v_\lambda$ strongly in $S^1_0(\Omega)$.
This, also considering both \eqref{eq:last} and \eqref{eq:gamma0}, gives
\[
I_\lambda(u_\lambda)<\gamma_0=\lim_{k\to+\infty}I_\lambda(v_k)=I_\lambda(v_\lambda),
\]
thus implying that $u_\lambda\not\equiv v_\lambda$. This closes the proof.
\end{proof}

\medskip

Gathering all the results established so far, we can finally provide the
\begin{proof}[Proof of Theorem \ref{thm:main}]
	Let $\Lambda$ be as in \eqref{eq:DefinitionLambda}, that is, 
    $$\Lambda := \sup \{ \lambda >0: \eqref{EqProblem} \textrm{ admits a weak solution}\}.$$
    From its very definition, this shows that \eqref{EqProblem}
    does not admit any weak solution for $\lambda > \Lambda$, which is assertion A).
    Moreover, combining Lemma \ref{lem:lambda0} with Lemma \ref{lem:LambdaInf}, we immediately have that $\Lambda\in (0,+\infty)$.
    Regarding the existence part, Theorem \ref{thm:Existence_First} proves that 
    there exists \emph{at least one weak solution} $u_\lambda$ of
    \eqref{EqProblem} (with $\lambda =\Lambda$), which is assertion B).
    As for assertion C),
    from Lemmas  \ref{lem:casoa1}, \ref{lem:casoa2}  and Theorem \ref{thm:casob},
    we got that a second
    solution always exists when $0<\lambda<\Lambda$.
\end{proof}


\medskip


\begin{thebibliography}{99}
	
	
	\bibitem{AbDiVa}
	B. Abdellaoui, A. Dieb, E. Valdinoci,
	{\em A nonlocal concave-convex problem with nonlocal mixed boundary data},
	Commun. Pure Appl. Anal. {\bf 17}(3), (2018), 1103--1120.
	
	\bibitem{AbTr}
	F. Abedin, G. Tralli,
	{\em Boundary regularity for subelliptic equations in the Heisenberg group},
	preprint. \url{https://arxiv.org/abs/2506.05151}
	
	
	\bibitem{Alama}
	S. Alama,
	{\em Semilinear elliptic equations with sublinear indefinite nonlinearities},
	Adv. Differential Equations {\bf 4}, (1999), 813--842.
	
%
\bibitem{ABC}
   A. Ambrosetti, H. Br\'{e}zis, G. Cerami, 
   {\em Combined effects of concave and convex nonlinearities in some elliptic problems},
   J. Funct. Anal. {\bf 122}(2), (1994), 519--543.


	
	
	
	
	\bibitem{BadTar}
	M. Badiale, G. Tarantello, 
	{\em Existence and multiplicity results for elliptic problems with critical growth and discontinuous nonlinearities}, 
	Nonlinear Anal. Theory Methods Appl. {\bf 29}, (1997), 639--677. 
	
	\bibitem{BaCiCu}
	A. Baldi, G. Citti, G. Cupini, 
	{\em Schauder estimates at the boundary for sublaplacians in Carnot groups}, 
	Calc. Var. Partial Differential Equations {\bf 58} (2019), art. 204.
	
	\bibitem{BaGaMu}
	A. Banerjee, N. Garofalo, I.H. Munive, 
	{\em Higher order boundary Schauder estimates in Carnot groups}, 
	Math. Ann. {\bf 390} (2024), 6013--6047.
	
	\bibitem{BESS}
	B. Barrios, E. Colorado, R. Servadei, F. Soria, 
	{\em A critical fractional equation with concave-convex
		power nonlinearities}, 
	Ann. Inst. H.Poincar\'{e} Anal. Non Lin\'{e}aire {\bf 32}(4), (2015), 875--900.
	
	
	\bibitem{BiGaVe}
	S. Biagi, M. Galeotti, E. Vecchi,
	{\em Critical singular problems in Carnot groups}, preprint. \url{https://arxiv.org/abs/2506.07521}
	
	\bibitem{BPV}
	S. Biagi, A. Pinamonti, E. Vecchi, 
	{\em Sublinear Equations Driven by Hörmander Operators}, J. Geom. Anal. {\bf 32}, (2022), art. 121.
	
	\bibitem{BV} 
	S. Biagi, E. Vecchi,
	{\em On the existence of a second positive solution to mixed local-nonlocal concave-convex critical problems}, Nonlinear Anal. {\bf 256}, (2025), 113795. 
	
	
	
	
	
	
	\bibitem{BLU}
	A. Bonfiglioli, E. Lanconelli, F. Uguzzoni, 
	{\em Stratified Lie Groups and Potential Theory for their
		Sub-Laplacians}. Springer Monographs in Mathematics, vol. 26. Springer, New York, NY (2007).
	
	\bibitem{BoUg}
	A. Bonfiglioli, F. Uguzzoni,
	{\em Nonlinear Liouville theorems for some critical
		problems on H-type groups},
	J. Funct. Anal. {\bf 207}(1), (2004), 161--215.
	
	
	
	
	
%
	
	\bibitem{BrezisKamin}
	H. Br\'{e}zis, S. Kamin, 
	{\em Sublinear elliptic equations in $\mathbb{R}^{n}$},
	Manuscripta Math. {\bf 74}(1), (1992), 87--106.
	
	
	\bibitem{BN}
	H. Br\'{e}zis, L. Nirenberg, 
	{\em Positive solutions of nonlinear elliptic equation involving the critical Sobolev exponent}, 
	Comm. Pure Appl. Math. {\bf 36}, (1983), 437--477. 
	
	
	
		\bibitem{BNH1C1}
		H. Br\'{e}zis, L. Nirenberg, 
		{\em $H^1$ versus $C^1$ local minimizers}, 
		C. R. Acad. Sci. Paris {\bf 317}, (1993), 465--472. 
	
		
	
	\bibitem{CCP}
	F. Charro, E. Colorado, I. Peral, 
	{\em Multiplicity of solutions to uniformly elliptic fully nonlinear equations with concave-convex right-hand side}, 
	J. Differential Equations {\bf 246}(11), (2009), 4221--4248.
	
	
	\bibitem{Citti}
	G. Citti,
	{\em Semilinear Dirichlet problem involving critical exponent for the Kohn Laplacian}, 
	Ann. Mat. Pura. Appl. {\bf 169}, (1995), 375--392.
	
	\bibitem{CitUg}
	G. Citti, F. Uguzzoni, 
	{\em Critical semilinear equations on the Heisenberg group: the effect of the
		topology of the domain}, Nonlinear Anal. {\bf 46}, (2001), 399--417.
	
	
	\bibitem{CoPe}
	E. Colorado, I. Peral, 
	{\em Semilinear elliptic problems with mixed Dirichlet-Neumann boundary conditions}, 
	J. Funct. Anal. {\bf 199}(2), (2003), 468--507.

	
	
	
	\bibitem{Ekeland}
	I. Ekeland,
	{\em On the variational principle},
	J. Math. Anal. Appl. {\bf 47}, (1974), 324--353.
	
	\bibitem{FelliUgu}
	V. Felli, F. Uguzzoni, 
	{\em Some existence results for the Webster scalar curvature problem in presence of symmetry},
	Ann. Mat. Pura Appl. {\bf 183}(4), (2004), 469--493.
	
	
	\bibitem{Folland}
	G.B. Folland,
	{\em Subelliptic estimates and function spaces on nilpotent Lie groups}, 
	Ark. Mat. {\bf 13}, (1975), 161--207.
	
	
	
	\bibitem{FollandStein}
	G.B. Folland, E.M. Stein,
	{\em Hardy {{Spaces}} on {{Homogeneous
				Groups}}}, Princeton University Press, Dec. 2020.
	
	
	
%
%
	
	
	\bibitem{GaLa}
	N. Garofalo, E. Lanconelli, 
	{\em Existence and nonexistence results for
		semilinear equations on the Heisenberg group},
	Indiana Univ. Math. J. {\bf 41}(1), (1992), 71--98.

\bibitem{GaroVa2}
	N. Garofalo, D. Vassilev, 
	{\em Regularity near the characteristic set in the non-linear Dirichlet problem and conformal geometry of sub-Laplacians on Carnot groups}, 
	Math Ann {\bf 318}, (2000), 453–-516.
	
	
	\bibitem{GaroVa}
	N. Garofalo, D. Vassilev, 
	{\em Symmetry properties of positive entire solutions of Yamabe-type equations on groups of Heisenberg type}, Duke Math. J. {\bf 106}(3), (2001), 411--448.
	
	
	
	
	
%
%
	
	\bibitem{Hormander}
	L. H\"{o}rmander, 
	{\em Hypoelliptic second order differential equations}, 
	Acta Math. {\bf 119}, (1967), 147--171.
	
	\bibitem{Jerison}
	D.S. Jerison, 
	{\em The Dirichlet problem for the Kohn Laplacian on the Heisenberg group. I}, 
	J. Functional Analysis {\bf 43}(1), (1981), 97--142.
	
	\bibitem{Jerison2}
	D.S. Jerison, 
	{\em The Dirichlet problem for the Kohn Laplacian on the Heisenberg group. II}, 
	J. Functional Analysis {\bf 43}(2), (1981), 224--257.
	
	
	\bibitem{JerisonLee}
	D.S. Jerison, J.M. Lee, 
	{\em The Yamabe problem on CR manifolds}, 
	J. Differential Geom. {\bf 25}(2), (1987), 167--197.
	
	\bibitem{JerisonLee2}
	D.S. Jerison, J.M. Lee, 
	{\em Extremals for the Sobolev inequality on the Heisenberg group and the CR Yamabe problem}, 
	J. Amer. Math. Soc. {\bf 1}(1), (1988), 1--13.

	\bibitem{JerisonLee3}
D.S. Jerison, J.M. Lee, 
{\em Intrinsic CR normal coordinates and the CR Yamabe problem},
J. Differ. Geom. {\bf 29}, (1989), 303--343.
%
	
	
	

	
%
%
	
	
	
	\bibitem{Loiudice1}
	A. Loiudice,
	{\em Semilinear subelliptic problems with critical growth
		on Carnot groups}, Manuscripta Math. {\bf 124}, (2007), 247--259.
	
	
	\bibitem{Loiudice2}
	A. Loiudice,
	{\em Critical growth problems with singular nonlinearities on Carnot groups},
	Nonlinear Anal. {\bf 126}, (2015), 415--436.
	
	\bibitem{Loiudice3}
	A. Loiudice,
	{\em Optimal decay of $p$-Sobolev extremals on Carnot groups},
	J. Math. Anal. Appl. {\bf 470}(1), (2019), 619--631.
	
	\bibitem{Loiudice4}
	A. Loiudice,
	{\em Critical problems with Hardy potential on stratified Lie groups},
	Adv. Differential Equations {\bf 28}(1-2), (2023), 1--33.
	
	
	
	
	\bibitem{MaMaPi}
	A. Maalaoui, V. Martino, A. Pistoia, 
	{\em Concentrating solutions for a sub-critical sub-elliptic problem},
	Differential Integral Equations {\bf 26}(11-12), (2013), 1263--1274.	
%
	
	\bibitem{MaUg}
	A. Malchiodi, F. Uguzzoni, 
	{\em A perturbation result for the Webster scalar curvature problem on the CR sphere},
	J. Math. Pures Appl. {\bf 81}(10), (2002), 983--997.
	
	%
	
	%
	
	
	

	
	\bibitem{PaPiTe}
	G. Palatucci, M. Piccinini, L. Temperini,	
	{\em Struwe's global compactness and energy approximation of the critical Sobolev embedding in the Heisenberg group},
	Adv. Calc. Var. (2024), \url{https://doi.org/10.1515/acv-2024-0044}
	
	
	
	\bibitem{Ruzhansky_Suragan}
    M. Ruzhansky, D. Suragan, 
    {\em Green’s Identities, Comparison Principle and Uniqueness of Positive Solutions for Nonlinear p-sub-Laplacian Equations on Stratified Lie Groups}, 
    Potential Anal {\bf 53}, (2020), 645--658.
	
	
	\bibitem{Struwe}
	M. Struwe, 
	{\em Variational Methods}, 
	Springer-Verlag, Berlin, 1990. xiv+244 pp.
	

	
	
	\bibitem{Tarantello}
	G. Tarantello, 
	{\em On nonhomogeneous elliptic equations involving critical Sobolev exponent}, 
	Ann. Inst. H. Poincar\'{e} C Anal. Non Lin\'{e}aire {\bf 9}, (1992), 281--304.
	
	\bibitem{Ugu1}
	F. Uguzzoni, 
	{\em A non-existence theorem for a semilinear Dirichlet problem involving critical exponent on halfspaces of the Heisenberg group}, 
	NoDEA Nonlinear Differential Equations Appl. {\bf 6}, (1999), 191--206.
	
	
	
	
	
\end{thebibliography}
 \end{document}